\documentclass[11pt]{article}

\usepackage[compress]{natbib}
\usepackage{amsmath,amssymb,amsthm}
\usepackage{color}
\usepackage[margin=1in]{geometry}
\usepackage{setspace}
\usepackage{amsthm, amsmath, amssymb, amsfonts}
\usepackage{subcaption}
\usepackage{bm}
\usepackage[dvipsnames]{xcolor}
\usepackage{hyperref}
\hypersetup{pdftex, colorlinks=true, citecolor=MidnightBlue, linkcolor=red!80!black, urlcolor=MidnightBlue}
\usepackage{caption}
\usepackage{url}            
\usepackage{booktabs}       
\usepackage{nicefrac}       
\usepackage{graphicx}
\usepackage{caption}
\usepackage{thm-restate}
\usepackage[ruled,linesnumbered]{algorithm2e}
\usepackage{algpseudocode}
\usepackage{bbm}
\usepackage[export]{adjustbox}

\usepackage{microtype}
\usepackage{mathpazo} 
\usepackage[scaled]{helvet} 
\usepackage{courier} 
\normalfont
\usepackage[T1]{fontenc}

\usepackage{threeparttable}
\allowdisplaybreaks[4] 
\usepackage[shortlabels]{enumitem} 

\usepackage{tikz}
\usetikzlibrary{fit}
\tikzset{%
	highlight/.style={rectangle,blend mode = multiply,draw=blue!90!black,thick,rounded corners = 0.3 mm,inner sep=0.5pt}
}

\usepackage{tcolorbox}
\def\O#1{\text{\ding{\the\numexpr#1+171}}}

\usepackage{mathrsfs}

\newcommand{\diff}{\mathrm{d}}

\newcommand{\E}{\mathbb{E}}
\newcommand{\EE}{\mathbb{E}}

\newcommand{\R}{\mathbb{R}}
\newcommand{\N}{\mathbb{N}}

\DeclareMathAlphabet{\mathpzc}{OT1}{pzc}{m}{it}

\newcommand{\mc}{\mathcal}

\newcommand{\opt}{^\star}

\renewcommand{\ge}{\geqslant}

\renewcommand{\geq}{\geqslant}
\renewcommand{\leq}{\leqslant}
\renewcommand{\epsilon}{\varepsilon}

\newcommand{\suchthat}{\mathrm{s.t.}}

\newcount\Comments
\newcommand{\kibitz}[2]{\ifnum\Comments=1{\textcolor{#1}{\textsf{\footnotesize #2}}}\fi}

\usepackage[nottoc]{tocbibind}
\usepackage{multirow}
\usepackage[english]{babel}
\usepackage{array}
\usepackage[nottoc]{tocbibind}
\usepackage{lettrine}
\usepackage[normalem]{ulem}

\usepackage{multirow}
\usepackage[english]{babel}
\usepackage{array}
\usepackage{enumitem}
\usepackage{lettrine}
\usepackage{centernot}
\usepackage{todonotes}
\usepackage{fontawesome}
\usepackage{booktabs}

\usepackage{cleveref}
\usepackage{tablefootnote}
\usepackage{dsfont} 
\usepackage{bbm}
\usepackage{amsthm}

\newtheorem{theorem}{Theorem}[section]
\newtheorem{lemma}[theorem]{Lemma}
\newtheorem{proposition}[theorem]{Proposition}
\newtheorem{assumption}[theorem]{Assumption}
\newtheorem{corollary}[theorem]{Corollary}

\theoremstyle{definition}
\newtheorem{definition}[theorem]{Definition}

\theoremstyle{remark}
\newtheorem{rmk}[theorem]{Remark}
\newcommand{\email}[1]{\href{mailto:#1}{\texttt{#1}}}

\title{Unifying Distributionally Robust Optimization\\  via Optimal Transport Theory}
\author{Jose Blanchet\thanks{Management Science and Engineering, Stanford University
  (\email{jose.blanchet@stanford.edu}).}
\and Daniel Kuhn\thanks{Risk Analytics and Optimization Chair, EPFL
  (\email{daniel.kuhn@epfl.ch}).}
\and Jiajin Li\thanks{Sauder School of Business, University of British Columbia
  (\email{jiajin.li@sauder.ubc.ca}).}
  \and Bahar Ta{\c s}kesen\thanks{Booth School of Business, University of Chicago
  (\email{bahar.taskesen@ chicagobooth.edu}).}}
\begin{document}
	\maketitle
\begin{abstract}
In recent years, two prominent paradigms have shaped distributionally robust optimization (DRO), modeling distributional ambiguity through $\phi$-divergences and Wasserstein distances, respectively. While the former focuses on ambiguity in likelihood ratios, the latter emphasizes ambiguity in outcomes and uses a transportation cost function to capture geometric structure in the outcome space. This paper proposes a unified framework that bridges these approaches by leveraging optimal transport (OT) with conditional moment constraints. Our formulation enables adversarial distributions to jointly perturb likelihood ratios and outcomes, yielding a generalized OT coupling between the nominal and perturbed distributions. We further establish key duality results and develop tractable reformulations that highlight the practical power of our unified approach.
\end{abstract}
\section{Introduction}
DRO has emerged as a powerful paradigm for decision-making under uncertainty, offering a principled defense against the adverse effects of distributional ambiguity \citep{blanchet2021statistical, chen2020distributionally, kuhn2019wasserstein, Kuhn_Shafiee_Wiesemann_2025, rahimian2019distributionally}. There are two popular ways to quantify distributional misspecification, each naturally giving rise to a distinct class of ambiguity sets. The first models corruption in the likelihoods of a baseline distribution (often the empirical distribution), capturing misspecification through likelihood ratios. This perspective leads to $\phi$-divergence-based uncertainty sets; see, for example, \citep{ref:bayraksan2015data, ben2009robust, bertsimas2018data, bertsimas2004price, breuer2016measuring, duchi2021learning, el2005robust, hansen2001robust, hu2013kullback, iyengar2005robust, jiang2018risk, lam2016robust, lam2019recovering, lim2006model, namkoong2016stochastic, namkoong2017variance, ref:van2021data, wang2016likelihood}. The second models potential corruptions directly in the outcomes themselves, leading to ambiguity sets defined via some Wasserstein distance, which quantifies misspecification in the space of realizations; see, e.g., \citep{blanchet2019robust, ref:blanchet2019quantifying, chen2022data, gao2017wasserstein, ref:gao2022distributionally, lee2018minimax, mohajerin2018data, pflug2007ambiguity, scarf1958min, shafieezadeh2019regularization, sinha2018certifying, xie2021distributionally}.

Traditionally, these $\phi$-divergence and Wasserstein distance approaches have been treated as distinct mechanisms, each tailored to a different source of distributional misspecification. In this paper, we show that they can be seamlessly unified within a generalized OT DRO framework on a lifted space with conditional moment constraints. This perspective enables OT methods to capture distributional uncertainty in both likelihoods and actual outcomes under a single, coherent model.

There are several advantages to our proposed unification under the OT umbrella. First, a substantial body of computational research has emerged in both the OT-based DRO community \citep{blanchet2021optimal, jiajin2021efficient, li2020fast, li2019first, sinha2018certifying, wang2021sinkhorn, yu2022fast} as well as in the broader OT community \citep{arjovsky2017wasserstein, cuturi2013sinkhorn, genevay2016stochastic, gulrajani2017improved, peyre2019computational, seguy2018large, ref:tacskesen2023semi}. Our framework has the potential to foster cross-fertilization and generate synergies between the techniques developed in these areas. Second, our approach inherits the natural strengths of the OT framework, which provides an optimal coupling between the nominal and adversarial distributions. This enhances interpretability and offers valuable modeling insight into how scenarios may shift to amplify their influence on risk relative to a performance criterion of interest. Finally, the OT-DRO literature offers a rich theoretical foundation that (we expect) can be readily adapted to our setting. This includes work on selecting optimal ambiguity radii \citep{blanchet2019robust, blanchet2021statistical, blanchet2022confidence, blanchet2019optimal}, analyzing statistical properties of OT-DRO problems \citep{an2021generalization, aolaritei2022performance, azizian2023exact, blanchet2023statistical, gao2022finite, he2021higher}, characterizing least favorable distributions and Nash equilibria \citep{ref:gao2022distributionally, mohajerin2018data, shafieezadeh2023new, yue2021linear, zhang2022wasserstein}, and exploring connections with regularization techniques \citep{blanchet2019robust, chen2018robust, cranko2021generalised, gao2017wasserstein, shafieezadeh2019regularization, shafieezadeh2015distributionally}, among others. We therefore expect that this work to stimulate new contributions across these areas from a unified perspective.

The paper is structured as follows. Section~\ref{sec:pre} introduces the problem setup and presents a new OT discrepancy with conditional moment constraints that gives rise to a lifted DRO formulation. Section~\ref{sec:mot-express} shows that existing DRO models with $\phi$-divergence, Wasserstein, and Sinkhorn ambiguity sets can all be recast as instances of this lifted model. Section~\ref{sec:duality} develops a duality theorem for the lifted DRO framework, addressing the challenges posed by the conditional moment constraints. Section~\ref{sec:cost_function} introduces a joint $\phi$-divergence and OT discrepancy that seamlessly interpolates between pure Wasserstein distances and pure $\phi$-divergences. This section additionally demonstrates that the corresponding DRO models also fall within the lifted formulation and develops tractable reformulations. Section~\ref{sec:visu} presents numerical results showing that the proposed approach can outperform classical DRO methods.

\textbf{Notation.}
The set of all integers up to~$n \in \mathbb N$ is denote by $[n] = \{1, \ldots, n\}$, and the sets of real numbers, nonnegative real numbers and extended real numbers are denoted by~$\R$, $\R_+$ and~$\overline{\R}$, respectively. We adopt the conventions of extended arithmetic, whereby $\infty \cdot 0=0 \cdot \infty=0 / 0=0$ and $\infty-\infty=-\infty+\infty=1 / 0=\infty$. Capital letters such as $Z$, $V$ or~$W$ stand for random variables, and calligraphic letters such as $\mc Z$, $\mc V$ or~$\mc W$ denote sets. The indicator function $\mathbf{1}_{\mc Z}$ of a set $\mc Z\subseteq \R^d$ is defined through $\mathbf{1}_{\mc Z}(z) = 1$ if $z\in\mc Z$; $=0$ otherwise. For any closed set $\mc Z \subseteq \R^d$, $\mc P(\mc Z)$ denotes the family of all Borel probability distributions on~$\mc Z$, and $\EE_{\mu}[\cdot]$ is the expectation operator with respect to~$\mu \in \mc P(\mc Z)$. For extended real-valued functions $f:\mc Z\rightarrow \overline\R$, we set $\EE_{\mu}[f]=\EE_{\mu}[\max(f,0)]+\EE_{\mu}[\min(f,0)]$, which is well-defined by our conventions of extended arithmetic. We say that $f$ is proper if $f(z)>-\infty$ for all $z\in \mc Z$ and if there exists $z \in \mc Z$ with $f(z) < \infty$. The convex conjugate of~$f$ is the function $f^*:\R^d \rightarrow \overline \R$ defined via $f^*(y) = \sup_{x\in\R^d}y^\top x-f(x)$. Finally, $\|\cdot\|$ denotes the Euclidean norm.

\section{Problem Setup}
\label{sec:pre}
%

We study data-driven decision problems under uncertainty. Each possible decision results in an uncertain loss $\ell(Z)$ characterized by a (decision-dependent) upper-semicontinuous loss function $\ell : \mc Z \to \overline \R$ and a (decision-{\em in}depen\-dent) vector $Z \in \mc Z$ of risk factors. Thus, we identify different loss functions~$\ell$ with different decisions. This abstract viewpoint obviates the need to introduce specific symbols for decisions and thus avoids clutter. The set of all feasible loss functions is denoted by $\mc L$. The risk associated with a decision $\ell \in \mc L$ is then defined as the expected loss $\mathbb{E}_{\mu_0}[\ell(Z)]$, where~$\mu_0$ denotes the probability distribution of~$Z$. In addition, the {\em optimal} risk is defined as the risk of the least risky admissible decision and thus equals 
\[
    \inf\limits_{\ell \in \mc L} \EE_{\mu_0}[\ell(Z)].
\]
In practice, the distribution~$\mu_0$ is often unknown, in which case it is impossible to evaluate the risk of any non-trivial decision exactly. Nevertheless, the decision-maker often has access to a finite number of training samples drawn from~$\mu_0$. In this case it is expedient to approximate the unknown true distribution~$\mu_0$ with the empirical distribution~$\hat\mu$, that is, the discrete uniform distribution on the training samples, and to approximate the true risk under~$\mu_0$ with the empirical risk under~$\hat\mu$. However, it is well-known that this sample average approximation can fall prey to the {\em optimizer's curse} or the {\em optimization bias}~\citep{shapiro2003monte}. This means that the risk of the optimal decision is optimistically biased. In other words, the optimal decision displays a higher risk on unseen test data than on the training data, which was used for its construction.

Nowadays there is burgeoning research on DRO that aims to mitigate the optimizer's curse. Unlike the sample average approximation, which places full trust in the empirical distribution~$\hat \mu$ and thus minimizes the empirical risk, DRO acknowledges that the unknown data-generating distribution is likely to differ from~$\hat\mu$ and thus adopts a worst-case perspective. Specifically, DRO seeks a decision that minimizing the worst-case risk, where the worst case is taken over all distributions in a prescribed ambiguity set~$\mc B$. The ambiguity set is usually constructed from the training data and is supposed to encompass the unknown true distribution $\mu_0$ with high confidence.

While DRO is often promoted as a remedy for the optimizer's curse, there are scenarios where the data-generating distribution diverges from the distribution over which the trained decision will be deployed. In such settings, there is an even greater urgency to find a decision that minimizes the worst-case risk across all distributions in a suitable ambiguity set. Formally, we will study the following DRO formulation.
\begin{equation}
    \inf\limits_{\ell \in \mc L}~\sup\limits_{\mu\in \mc B} \EE_{ \mu} [\ell(Z)]
    \label{eq:dro}
\end{equation}
To be able to control the model's conservativeness, we typically define the ambiguity set~$\mc B$ as a neighborhood ball $\mathcal{B}_{\mathds D}^r(\hat{\mu}) = \{\mu \in \mathcal{P}(\mathcal{Z}): \mathds D(\mu, \hat{\mu}) \leq  r\}$ of radius~$r\geq 0$ around some nominal distribution~$\hat{\mu}$ ({\em e.g.}, an empirical distribution). This ball is defined with respect to a discrepancy measure~$\mathds D$ on~$\mc P(\mc Z)$. Hence, any instance of the DRO problem~\eqref{eq:dro} is comprehensively encoded by the tuple $ (\mc L, \mc Z, \hat \mu, \mathds D, r)$. Similarly, the instance of the inner worst-case expectation problem in~\eqref{eq:dro} corresponding to a particular loss function $\ell\in\mc L$ can be encoded by the tuple $(\ell, \mc Z, \hat \mu, \mathds D, r)$.

In this paper, we lift the original domain~$\mathcal Z$ of the uncertain risk factors to a higher-dimensional space $\mathcal{V} \times \mathcal{W}$. We then introduce a new OT discrepancy on $\mathcal{V} \times \mathcal{W}$ with conditional moment constraints on $\mc W$, which will allow us to re-interpret~$\mathcal{B}$ as an OT-based ambiguity set. To initiate our discussion, we first formally introduce the notion of an OT discrepancy with conditional moment constraints.

\begin{definition}[OT discrepancy with conditional moment constraints]
\label{def:mot}
If $\mc V\subseteq \R^{d_v}$ and $\mc W\subseteq \R_+$ are convex and closed sets, $\mc G$ is a countably generated sub‐$\sigma$‐field of the Borel $\sigma$-field on $(\mc V\times\mc W)^2$ and $c: (\mc V \times \mc W)^2 \rightarrow  (-\infty,+\infty]$ is lower semicontinuous and bounded below, then the OT discrepancy on $\mc P(\mc V\times \mc W)$ with conditional moment constraints induced by $c $ and $ \mc G$ is the function $\mathds M: \mc P(\mc V \times \mc W)^2 \to (-\infty,+\infty]$ defined~via
\begin{equation}
\label{eq:mot}
\mathds M(\nu, \hat \nu) = \left\{
\begin{array}{clll}
\inf &\EE_{\pi}[c((V , W), (\hat V, \hat W))]\\
\suchthat  &\pi \in \mc P((\mc V \times \mc W)^2)\\[-0.7ex]
& \pi_{(V, W)}  = \nu,\ \pi_{(\hat V, \hat W)}  = \hat \nu \\
& \EE_\pi[W | \mc G] = 1 \quad \pi \text{-a.s,}
\end{array}\right.
\end{equation}
where $\pi_{(V, W)}$ and $\pi_{(\hat V, \hat W)}$ are the marginal distributions of $(V, W)$ and $(\hat V, \hat W)$ under~$\pi$. 
\end{definition}

The generalized OT problem~\eqref{eq:mot} introduced in Definition \ref{def:mot} is not always feasible. However, in the remainder of the paper we will focus on particular choices of~$\hat \nu$, $\nu$, $c$ and~$\mc G$ for which the problem is guaranteed to be feasible.

\begin{rmk}
If $\mc G = \{\emptyset,(\mc W\times\mc V)^2 \}$ is the trivial $\sigma$-field, the conditional moment constraint $\EE_\pi[W | \mc G] = 1$ reduces to the deterministic constraint $\E_\pi[W]=1$. Similarly, if $\mc G = \sigma(\hat V)$ is generated by the random variable $\hat V$, then the constraint $\EE_\pi[W | \mc G] = 1$ is equivalent to $\EE_\pi[W | \hat V] = 1$, and if $\hat V$ is supported on $n$~points, this uncertain constraint decomposes into $n$ deterministic constraints--one per support point.
\end{rmk}

The OT discrepancy $\mathds M$ gives rise to a neighborhood ball $\mc B_{\mathds M}^r(\hat \nu)=\{\nu\in\mathcal P(\mc V\times\mc W): \mathds M(\nu,\hat\nu)\leq r\}$ of radius $r\geq 0$ around a prescribed nominal distribution~$\hat\nu \in \mc P(\mc V \times \mc W)$. Given an upper semicontinuous and $\hat \nu$-integrable loss function~$f$, we can then introduce the following OT-DRO problem with conditional moment constraints.
\begin{equation}
    \sup\limits_{\nu \in\mc B_{\mathds M}^r(\hat \nu)} \EE_\nu [f(V, W)]
    \label{eq:mot-dro}
\end{equation}
Note that~\eqref{eq:mot-dro} is fully determined by the tuple~$(f, \mc V, \mc W, \mc G, \hat \nu, c, r)$. The following mild integrability and growth conditions on~$c$ will be adopted throughout the paper.

\begin{assumption}[Integrability and growth conditions on~$c$]
\label{ass:growth}
There are constants $\delta\in(0,1]$ and $k>0$, and a Borel function $\varphi:\mathcal W\to\R_+$ such that the following hold.

\begin{enumerate}[label=\normalfont(\alph*)]
    \item \textit{Integrability along $\hat w$:} 
          $\displaystyle \mathbb E_{\hat\nu}[\varphi(\hat W)]<\infty$;
    \item \textit{One-sided super-linear growth along $w$:}
          \[
              c\bigl((v,w),(\hat v,\hat w)\bigr)
              \;\ge\;
              k\,\lVert w\rVert^{1+\delta}-\varphi(\hat w)
              \quad
              \forall\,(v,w),(\hat v,\hat w)\in\mathcal V\times\mathcal W.
          \]
\end{enumerate}
\end{assumption}

\Cref{ass:growth} ensures that the infimum of the generalized OT problem~\eqref{eq:mot} is attained if $\mathds M(\nu, \hat \nu)< \infty$ (see \Cref{lemma:mot-attainment} in the appendix). Thus, it allows us to rewrite problem~\eqref{eq:mot-dro} as the following explicit optimization over probability couplings:
\begin{equation}
    \begin{array}{cclll}
    &\sup &\EE_{\pi}[f(V, W)]\\
    &\suchthat & \pi \in \mc P ((\mc V \times \mc W)^2)\\[-0.7ex]
    &&\pi_{(\hat V, \hat W)}  = \hat \nu\\
    &&  \EE_\pi[W | \mc G] = 1 \quad \pi \text{-a.s.} \\
    && \EE_\pi[c((V, W), (\hat V, \hat W))] \leq r. 
    \end{array}
    \label{eq:mot-dro-extended}
\end{equation}
In the remainder, we showcase several advantages of the OT-DRO model with conditional moment constraints. In Section~\ref{sec:mot-express}, we demonstrate that problem~\eqref{eq:mot-dro} offers sufficient flexibility to provide a unified DRO model, which generalizes a wide range of existing DRO models based on probability discrepancies. The key idea is to identify a lifting map that connects the tuple $(\ell, \mc Z, \hat \mu, \mathds D, r)$ characterizing the inner problem in~\eqref{eq:dro} to the tuple $(f, \mc V, \mc W, \mc G, \hat \nu, c, r)$ characterizing an instance of~\eqref{eq:mot-dro}. Section~\ref{sec:duality} then develops a duality theory for problem~\eqref{eq:mot-dro}. Finally, in Section~\ref{sec:cost_function}, we show that problem~\eqref{eq:mot-dro} gives rise to previously unexplored DRO models that seamlessly interpolate between {$\phi$}-divergence and Wasserstein distance-based DRO models.

\section{Expressiveness of OT-DRO with Conditional Moment Constraints}
\label{sec:mot-express}

In the following, we show how the proposed OT-DRO model~\eqref{eq:mot-dro} with conditional moment constraints serves as a unifying framework for a broad class of DRO problems. In particular, it encompasses Wasserstein DRO \citep{ref:wozabal2012framework, mohajerin2018data, ref:zhao2018data, ref:blanchet2019quantifying, ref:gao2022distributionally}, generalized $\phi$-divergence DRO~\citep{ref:bayraksan2015data,namkoong2016stochastic,ref:van2021data}, and Sinkhorn DRO~\citep{wang2021sinkhorn,azizian2022regularization,dapogny2022entropy}. Throughout this section, we define the nominal probability distribution~$\hat \mu$ as the discrete empirical distribution over $n$ training samples $\{\hat z_i\}_{i=1}^n$. Specifically, we define
$\hat\mu= \frac{1}{n}\sum_{i=1}^n \delta_{\hat z_i}$,
where $\delta_{\hat z_i}$
denotes the Dirac distribution that concentrates unit mass at the atom $\hat z_i \in \mc Z$.

\subsection{Generalized \texorpdfstring{$\phi$}~-Divergence-Based DRO}
We first define the notion of an entropy function, which is essential to define a generalized $\phi$-divergence. 
\begin{definition}[Entropy function]
    An entropy function $\phi: \R \to \overline\R$ is any lower semicontinuous convex function with $\phi(1) =0$ and $\rm{dom}(\phi) \subseteq \R_+$. The speed of growth of $\phi$ at  $+\infty$ is defined as $\phi_\infty' = \lim_{t \to +\infty} \phi(t)/t \in \overline{\R}$.
\end{definition}
In DRO research, $\phi$-divergences are usually defined as discrepancy measures between two probability distributions $ \mu $ and $ \hat \mu $ with the property that $ \hat \mu $ is absolutely continuous with respect to $ \hat \mu $ (denoted as $\hat\mu\ll\mu$). In this paper, we study generalized $ \phi $-divergences that do not impose any absolute continuity conditions on the probability distributions. Consequently, the resulting divergence ambiguity sets include distributions that assign positive mass to atoms not present in the support of $ \hat \mu $.

\begin{definition}[Generalized $\phi$-divergence~\citep{ali1966general,csiszar1964informationstheoretische,ref:csiszar1967information}]
\label{defi-phi}
    The generalized $\phi$-di\-ver\-gence between two probability distributions $\mu, \hat\mu \in \mc P(\mc Z)$ is defined as
    \begin{equation*}
        \mathds D_\phi(\mu, \hat \mu) = \int_{\mc Z}
        \frac{\diff \hat \mu}{\diff \rho}(z)  \cdot  \phi\left( \frac{\frac{\diff \mu}{\diff \rho}(z) }{ \frac{\diff \hat \mu}{\diff \rho}(z)}\right) \diff \rho(z),
     \end{equation*}
     where $\phi$ is an arbitrary entropy function, and~$\rho$ is a dominating measure of~$\mu$ and~$\hat \mu$ satisfying $\mu \ll \rho$ and $\hat \mu \ll \rho$. In addition, we use the conventions $\phi(0) = \lim_{t\rightarrow 0^+} \phi(t)$, $0\cdot\phi(\tfrac{0}{0}) =0$ and $0\cdot\phi(\tfrac{\alpha }{0}) =\alpha \cdot \lim_{t \to +\infty} \frac{\phi(t)}{t} = \alpha\cdot \phi'_\infty$ for any constant $\alpha >0$. 
\end{definition}
A dominating measure $\rho$ always exist. For example, one may set $\rho=\mu+\hat\mu$. In Remark~\ref{rem:Lebesgua-decomposition} below we will see that $\mathds D_\phi(\mu, \hat \mu)$ is in fact independent of the choice of~$\rho$.

\begin{rmk}
A further impetus for adopting the generalized $ \phi $-divergence is the relationship $\mathds{D}_\phi(\mu, \hat\mu) = \mathds{D}_\psi(\hat \mu, \mu)$, where $ \psi $ denotes the Csiszár dual of $ \phi$ and is defined through $ \psi(t) = t \cdot \phi(\tfrac{1}{t}) $ for all $t>0$, $\psi(0)=\phi'_\infty$ and $\psi'_\infty=\phi(0)$ \citep[Proposition 2]{sason2018f}. This relationship suggests that, even though $\phi$-divergences are generically asymmetric, we may focus on divergence ambiguity sets of the form $\{\mu \in \mathcal{P}(\mathcal{Z}): \mathds D(\mu, \hat{\mu}) \leq  r\}$, with the nominal distribution being the second argument of the divergence. 
\end{rmk}
Table~\ref{tab:phi-divergence} summarizes several common $\phi$-divergences and their Csiszár duals. 
\begin{table}[h!]
\setlength\tabcolsep{3pt}
\small
    \centering
\begin{tabular}{l l l c c}
\toprule
        Divergence & $\phi(t)$ & $\psi(t)$ &$\phi_\infty' = \infty \textbf{?}$ &$\psi_\infty' = \infty \textbf{?}$ \\ \hline
        Kullback-Leibler & $t \cdot \log(t) -t +1$ &$-\log(t) + t - 1$ & \faCheck &  \faTimes\\
        Burg Entropy & $-\log(t) + t - 1$ &  $t \cdot \log(t) -t +1$ &\faTimes & \faCheck\\
        $J$-divergence &$(t-1) \log(t)$ & $(t-1) \log(t)$ &\faCheck & \faCheck \\
        $\chi^2$-distance &$\frac{1}{t}(t-1)^2$ & $(t-1)^2$& \faTimes & \faCheck\\
        Modified $\chi^2$-distance &$(t-1)^2$ &$\frac{1}{t}(t-1)^2$ &\faCheck & \faTimes\\ 
        Hellinger distance & $(\sqrt{t}-1)^2$ &$(\sqrt{t}-1)^2$&\faTimes &\faTimes\\
        $\chi$-divergence of order~$m>1$ & $|t-1|^m$ &$t|\frac{1}{t}-1|^m$ &\faCheck & {\faTimes}\\
        Total variation distance &$|t-1|$ &$|t-1|$&\faTimes & {\faTimes}\\
        \bottomrule
        \end{tabular}
    \caption{Examples of $\phi$-divergences and their Csiszár duals. }
\label{tab:phi-divergence}
\vspace{-3mm}
\end{table}

Next, we will show that the generalized $\phi$-divergence-based DRO problem
\begin{equation}
    \sup\limits_{\mu \in \mc P(\mc Z)} \left\{ \EE_\mu[\ell(Z)] : \mathds D_\phi(\mu, \hat \mu)\leq r\right\}
    \label{eq:phi-div-dro}
\end{equation}
can be recast as an instance of the OT-DRO problem~\eqref{eq:mot-dro} with conditional moment constraints. This reformulation is based on a lifting of the original space~$\mc Z$ of uncertain parameters and the construction of an augmented nominal probability distribution.

The next lemma offers a new decomposition formula for generalized $\phi$-divergences. 

\begin{lemma}[Decomposition of generalized $\phi$-divergences]
The generalized $\phi$-di\-ver\-gence between two probability distributions $\mu, \hat\mu \in \mc P(\mc Z)$ can be decomposed as
\begin{equation}
\label{eq:phi-decom}
    \mathds D_\phi(\mu, \hat \mu)  = \displaystyle\int_{\mc Z}
    \phi\left( \frac{\frac{\diff \mu}{\diff \rho}(z) }{ \frac{\diff \hat \mu}{\diff \rho}(z)}\right) \diff \hat\mu(z) +  \phi_\infty' \cdot \mu \left(\frac{\diff \hat \mu}{\diff \rho}(z)=0\right),
\end{equation}
where~$\rho$ is any dominating measure of~$\mu$ and~$\hat \mu$ satisfying $\mu \ll \rho$ and $\hat \mu \ll \rho$.
\label{lemma:decomposition}
\end{lemma}

\begin{proof}
By Definition \ref{defi-phi},  we have
\begin{align*}  
    \mathds D_\phi(\mu, \hat \mu) = & \int_{{\mc Z}_+}
    \frac{\diff \hat \mu}{\diff \rho}(z)  \cdot \phi\left(\frac{\frac{\diff \mu}{\diff \rho}(z) }{ \frac{\diff \hat \mu}{\diff \rho}(z)}\right) \diff \rho(z) + \int_{{\mc Z}_0}
    \frac{\diff \hat \mu}{\diff \rho}(z)  \cdot \phi\left(\frac{\frac{\diff \mu}{\diff \rho}(z) }{ \frac{\diff \hat \mu }{\diff \rho}(z)}\right) \diff \rho(z) \\
      = &  \int_{\mc Z}
     \phi\left(\frac{\frac{\diff \mu}{\diff \rho}(z) }{ \frac{\diff \hat \mu}{\diff \rho}(z)}\right) \diff {\hat \mu }(z) +  \int_{{\mc Z}_0}\phi_\infty'\cdot \frac{\diff \mu}{\diff \rho}(z)\,  \diff \rho(z) \\
     =  &  \int_{\mc Z}
     \phi\left(\frac{\frac{\diff \mu}{\diff \rho}(z) }{ \frac{\diff \hat \mu}{\diff \rho}(z)}\right) \diff {\hat \mu }(z) +  \phi_\infty' \cdot \mu({\mc Z}_0),
\end{align*}
where ${\mc Z}_+ = \left \{z \in \mc Z :{\diff \hat \mu}/{\diff \rho}(z)>0\right \}$ and ${\mc Z}_0 = \left\{z \in \mc Z :{\diff \hat \mu}/{\diff \rho}(z)=0\right \}$. Here, the second equality exploits our standard conventions from Definition~\ref{defi-phi}.
\end{proof}

\Cref{lemma:decomposition} implies that if~$ \phi'_\infty =\infty$, then our generalized definition of a $\phi$-divergence collapses to the traditional one, which is given by
\[
    \mathds D_{\phi}(\mu, \hat\mu) = \begin{cases} \displaystyle\int_{\mc Z}\phi\left(\frac{\diff \mu}{\diff \hat \mu }(z)\right) \diff \hat \mu(z) & \text{if } \mu \ll \ \hat \mu,  \\ +\infty & \text{otherwise}. \end{cases}  
\]

\begin{rmk} [Lebesgue decomposition of $\mu$]
\label{rem:Lebesgua-decomposition}
\Cref{lemma:decomposition} is closely related to the classical Lebesgue decomposition theorem in measure theory. Indeed, the generalized $\phi$-divergence
\eqref{eq:phi-decom} can be reformulated as
\begin{equation}
\label{eq:phi-decom-without-rho}
\mathds D_\phi(\mu, \hat \mu) = \int_{\mc Z} \phi\left(\frac{\diff \mu^c}{\diff \hat \mu}(z)\right) \diff \hat \mu (z) + \phi_\infty' \cdot \mu^\perp(\mc Z),
\end{equation}
where~$\mu^c$ and~$\mu^\perp$ denote the probability measure~$\mu$ restricted to the sets~$\mc Z_+$ and~$\mc Z_0$ from the proof of \Cref{lemma:decomposition}, respectively. By construction, $\mu^c$ is thus absolutely continuous with respect to~$\hat{\mu}$, while $\mu^\perp$ is singular with respect to $\hat{\mu}$. Consequently, the expression $\mu = \mu^c+\mu^\perp$ represents the Lebesgue decomposition of $\mu$ relative to $\hat{\mu}$.
\end{rmk}

We can now use the insights from \Cref{lemma:decomposition} and Remark~\ref{rem:Lebesgua-decomposition} to characterize the worst-case distributions of the $\phi$-divergence-based DRO problem~\eqref{eq:phi-div-dro}.
\begin{proposition} [Worst-case distributions of problem \eqref{eq:phi-div-dro}]
\label{prop:phi-worst-case}
If $\mc Z$ is compact, then problem~\eqref{eq:phi-div-dro} is solved by a distribution~$\mu^\star$ supported on $\hat{\mc Z}=\{\hat z_i\}_{i=1}^{n+1}$, where $\hat z_i$ is the $i$-the atom of~$\hat\mu$, $i\in[n]$, and $\hat z_{n+1} \in \mathop{\arg\max}_{z \in\mc Z} \ell(z)$ is a worst-case scenario. 
\end{proposition}

\begin{proof}
We may assume without loss of generality that the solution $\hat z_{n+1}$ of the maximization problem $\max_{z\in \mc Z}\ell(z)$ is unique and that $\hat z_{n+1}\neq \hat z_i$ for all $i\in[n]$. By the decomposition~\eqref{eq:phi-decom-without-rho}, the $\phi$-divergence-based DRO problem~\eqref{eq:phi-div-dro} is equivalent to
\begin{equation}
    \sup\limits_{\mu \in\mc P (\mc Z)} \left\{ \int_{\mc Z}\ell (z)\,\diff\mu(z) \;:\; \int_{\mc Z} \phi\left(\frac{\diff \mu^c}{\diff \hat \mu}(z)\right) \diff \hat\mu(z) +  \phi_\infty' \cdot \mu^\perp(\mc Z)\leq r\right\}.
    \label{eq:phi-div-dro-decomposed}
\end{equation}
As $\mc Z$ is compact, the family $\mathcal{P}(\mc Z)$ of all probability distributions on~$\mc Z$ is compact with respect to the weak topology. In addition, as $\ell$ is upper semicontinuous on the compact set~$\mc Z$, its integral is weakly upper semicontinuous in $\mu$ thanks to Fatou's lemma. Thus, an optimal solution~$\mu\opt$ to problem~\eqref{eq:phi-div-dro} is guaranteed to exist. Below, we will individually analyze the cases~$\phi'_\infty < \infty$ and $\phi'_\infty  = +\infty$.
\paragraph{Case~I ($\phi'_\infty = +\infty$)} The $\phi$-divergence constraint in~\eqref{eq:phi-div-dro-decomposed} ensures that ${\mu\opt}^\perp(\mc Z)=0$ and that $\mu\opt$ is absolutely continuous with respect to the nominal distribution~$\hat \mu$. Hence, $\mu\opt$ is supported on the $n$ atoms $\{\hat z_i\}_{i=1}^n$ of $\hat\mu$, all of which belong to~$\hat{\mc Z}$. 
\paragraph{Case~II ($\phi'_\infty < \infty$)} 
We may assume without loss of generality that ${\mu\opt}^\perp(\mc Z) > 0$ for otherwise Case~I prevails. 
By the definition of the Lebesgue decomposition, ${\mu\opt}^c$ is supported on $\{\hat z_i\}_{i=1}^n$, while ${\mu\opt}^\perp(\{\hat z_i\}_{i=1}^n)=0$. To show that ${\mu\opt}^\perp$ must be supported on the point~$\hat z_{n+1}$, we assume for the sake of argument that ${\mu\opt}^\perp(\{\hat z_{n+1}\})<{\mu\opt}^\perp(\mc Z)$, and we construct a new probability measure $\tilde \mu={\mu\opt}^c+{\mu\opt}^\perp(\mc Z) \cdot \delta_{\hat z_{n+1}}$. By construction, we have ${\tilde \mu}^c={\mu\opt}^c$ and ${\tilde \mu}^\perp(\mc Z)={\mu\opt}^\perp(\mc Z)$, which implies via~\eqref{eq:phi-decom-without-rho} that $\mathds D_\phi(\tilde \mu, \hat \mu) = \mathds D_{\phi}(\mu\opt, \hat \mu)\leq r$. Hence, $\tilde \mu$ is feasible in~\eqref{eq:phi-div-dro-decomposed}. In addition, as ${\mu\opt}^\perp(\{\hat z_{n+1}\})<{\mu\opt}^\perp(\mc Z)$ and as $\hat z_{n+1}$ is the unique maximizer of $\ell$ over $\mc Z$, it is clear that 
\begin{align*}
    \int_{\mc Z}\ell (z)\,\diff\mu\opt(z) & = \int_{\mc Z} \ell (z) \diff {\mu\opt}^c (z) + \int_{\mc Z} \ell (z) \diff {\mu\opt}^\perp (z) \\ 
    & < \int_{\mc Z} \ell (z) \diff {\mu\opt}^c (z) + {\mu\opt}^\perp(\mc Z) \cdot \ell(\hat z_{n+1}) = \int_{\mc Z}\ell (z)\,\diff\tilde\mu(z),
\end{align*}
which contradicts the optimality of $\mu\opt$. Hence, our assumption was false, and we must have ${\mu\opt}^\perp(\{\hat z_{n+1}\})={\mu\opt}^\perp(\mc Z)$, that is, ${\mu\opt}^\perp$ is supported on $\{\hat z_{n+1}\}$. From this we may finally conclude that $\mu\opt$ is supported on~$\hat{\mc Z}$, and thus the claim follows.
\end{proof}

\begin{rmk}[Worst-case distributions of problem \eqref{eq:phi-div-dro}] 
\Cref{prop:phi-worst-case} reveals that the worst-case distribution of problem~\eqref{eq:phi-div-dro} is supported on the atoms $\{\hat z_i\}_{i=1}^n$ of $\hat\mu$ and---provided that $\phi_\infty'< \infty$---some worst-case scenario $\hat z_{n+1}\in\arg\max_{z\in\mc Z}\ell(z)$. From the proof of \Cref{prop:phi-worst-case} it is also evident that if $r$ exceeds $\phi(0) + \phi_\infty'$, then problem \eqref{eq:phi-div-dro} is solved by a distribution entirely supported on~$\arg\max_{z \in \mc Z} \ell(z)$. Table~\ref{tab:phi-divergence} highlights the divergences that possess a finite asymptotic speed of growth~$\phi'_\infty < \infty$. \Cref{prop:phi-worst-case} extends the results for finite probability spaces reported in \citep{ref:bayraksan2015data}.
\end{rmk}

The following main theorem shows that if $\mathds D=\mathds D_\phi$ is a generalized $\phi$-divergence, then the worst-case expectation problem in~\eqref{eq:dro} can be recast as an instance of the proposed OT-DRO model with conditional moment constraints \eqref{eq:mot-dro}.

\begin{theorem}[Generalized $\phi$-divergence-based DRO]
\label{theorem:phi-div-mot} Let $(\ell, \mc Z, \hat \mu, \mathds D, r)$ be an instance of the worst-case expectation problem in~\eqref{eq:dro} with $\mc Z$ compact, $\hat\mu=\! \frac{1}{n}\sum_{i=1}^n \delta_{\hat z_i}$ and $\mathds D=\mathds D_\phi$. Define $\mathcal{V} = \mathcal{Z}$, $\mathcal{W}=\mathbb{R}_+$, ${\mc G} = \{\emptyset,(\mc W\times\mc V)^2 \}$,  $f(v,w)=\ell(v)\cdot w$ and
\begin{equation*}
\label{eq:ref_measure_phi}
    \hat{\nu} = (1-\epsilon)\cdot \frac{1}{n} \sum_{i=1}^n \delta_{\left(\hat{z}_i,\frac{1}{1-\epsilon}\right)} + \epsilon \cdot \delta_{\left(\hat{z}_{n+1},0\right)}
\end{equation*}
for some $\epsilon\in(0,1)$ and $\hat z_{n+1} \in \mathop{\arg\max}_{z \in\mc Z} \ell(z)$.
In addition, set
\begin{equation}
    c((v, w), (\hat{v},\hat{w})) = \left\{ \begin{array}{ll} g(v) \cdot \phi\left(\frac{w}{g(v)}\right) & \text{if }v = \hat{v}\in \{\hat z_i\}_{i=1}^{n+1},\\
    \infty & \text{otherwise,}
    \end{array}\right.
\label{eq:cost-phi-mot-equiv}
\end{equation}
where $g(v)=\frac{1}{1-\epsilon}$ if $v\in \{\hat{z}_i\}_{i=1}^n$ and $g(v)=0$ if $v = \hat{z}_{n+1}$. 
Then, the inner worst-case expectation problem in~\eqref{eq:dro} has the same optimal value as problem~\eqref{eq:mot-dro}.
\end{theorem}




\begin{proof}
As $\mathds D=\mathds D_\phi$, the worst-case expectation problem in~\eqref{eq:dro} is equivalent to~\eqref{eq:phi-div-dro}. By \Cref{prop:phi-worst-case}, we may restrict~\eqref{eq:phi-div-dro} to the discrete distributions supported on~$\hat{\mc Z}=\{\hat z_i\}_{i=1}^{n+1}$ without sacrificing optimality. Any $\mu\in\mc P(\hat{\mc Z})$---and in particular $\hat\mu$---is absolutely continuous with respect to $\rho = \tfrac{1-\epsilon}{n} \sum_{i=1}^n \delta_{\hat{z}_i}+\epsilon \cdot \delta_{\hat z_{n+1}}$. By the definition of the generalized $\phi$-divergence, problem~\eqref{eq:phi-div-dro} is thus equivalent to
\begin{equation}
\label{eq:phi-1}
\begin{aligned}
    \max \limits_{\mu \in \mc P (\hat{\mc Z})}\left\{\EE_\rho\!\left[ \ell(Z) \cdot \frac{\diff \mu}{\diff \rho}(Z)\right]
    \;:\; \EE_\rho\!\left[ \frac{\diff \hat \mu}{\diff \rho}(Z) \cdot \phi\left( \frac{\frac{\diff \mu}{\diff \rho}(Z) }{ \frac{\diff \hat \mu}{\diff \rho}(Z)}\right)\right] \leq r\right\}.
\end{aligned}
\end{equation}
The remainder of the proof proceeds in two steps. By introducing an auxiliary random variable~$W$ to replace the likelihood ratio $\frac{\diff \mu}{\diff \rho}(Z)$, we will first demonstrate that problem~\eqref{eq:phi-1} can be reformulated as an optimization problem over probability measures in~$\mc P(\hat{\mc Z} \times \mc W)$ (Step~1). Next, we will show that the resulting reformulation can be lifted to an optimization problem over couplings in $\mc P((\mc V \times \mc W)^2)$ (Step~2). 

\noindent\textit{Step 1.} We start by reformulating~\eqref{eq:phi-1} as an optimization problem over probability distributions $\nu\in \mathcal{P}(
\hat{\mathcal{Z}} \times \mathcal{W})$, where $\mathcal{W} = \mathbb{R}_+$ represents a space of likelihood ratios.
\begin{equation}
    \begin{array}{ccll}
         & \max  & \EE_{\nu}\left[\ell(Z)\cdot W \right]\\
         &\suchthat  &\nu \in \mc P(\hat{\mc Z} \times \mc W)\\[-0.5ex] 
         && {\nu}_Z = \rho\\ 
         && \EE_{\nu}\left[W\right] = 1  \\
         && \EE_{\nu}\left[ \frac{\diff \hat \mu}{\diff \rho}(Z)\cdot \phi\left( \frac{W}{\frac{\diff  \hat \mu}{\diff \rho}(Z)}\right) \right] \leq r
    \end{array}
\label{eq:mot-phi-refor-1}
\end{equation} 
We will first show that for any~$\mu$ feasible
in~\eqref{eq:phi-1} there is~$\nu$ feasible in~\eqref{eq:mot-phi-refor-1} with the same objective function value. To this end, choose any~$\mu$ feasible in~\eqref{eq:phi-1} and define
\[
    \nu = (1-\epsilon)\cdot \frac{1}{n}\sum_{i=1}^n \delta_{\left(\hat z_i,\frac{\diff \mu}{\diff \rho}(\hat z_i) \right)} +\epsilon \cdot \delta_{\left(\hat z_{n+1},\frac{\diff \mu}{\diff \rho}(\hat z_{n+1}) \right)}. 
\]
By construction, the discrete distribution $\nu$ of $(Z,W)$ belongs to $\mc P(\hat{\mc Z} \times \mc W)$. In addition, the marginal distribution of~$Z$ under~$\nu$ coincides with~$\rho$, implying that $\nu$ satisfies the marginalization constraint in~\eqref{eq:mot-phi-refor-1}. As the conditional distribution of~$W$ given~$Z$ under~$\nu$ coincides with the Dirac distribution at~$\frac{\diff \mu}{\diff \rho}(Z)$,
we further have
\begin{equation*}
    \EE_{\nu}\left[W\right] =  \EE_\nu\!\left[ \frac{\diff \mu}{\diff \rho}(Z) \right]=  \EE_\rho\!\left[ \frac{\diff \mu}{\diff \rho}(Z) \right] = 1,
\end{equation*}
that is, $\nu$ satisfies the normalization constraint in~\eqref{eq:mot-phi-refor-1}. Similarly, we find
\[
    \EE_{ \nu}\left[ \frac{\diff \hat \mu}{\diff \rho}(Z)\cdot \phi\left( \frac{W}{\frac{\diff  \hat \mu}{\diff \rho}(Z)}\right) \right]
    = \EE_\rho\!\left[ \frac{\diff \hat \mu}{\diff \rho}(Z) \cdot \phi\left( \frac{\frac{\diff \mu}{\diff \rho}(Z) }{ \frac{\diff \hat \mu}{\diff \rho}(Z)}\right)\right] \leq r,
\]
where the inequality holds because~$\mu$ is feasible in~\eqref{eq:phi-1}. Hence, $\nu$ satisfies also the divergence constraint in~\eqref{eq:mot-phi-refor-1}. Finally, an analogous reasoning reveals that the objective value of~$\mu$ in~\eqref{eq:phi-1} matches that of~$\nu$ in~\eqref{eq:mot-phi-refor-1}. As $\mu$ was chosen arbitrarily, this implies that the optimal value of~\eqref{eq:mot-phi-refor-1} provides an upper bound on that of~\eqref{eq:phi-1}.

Next, we show that for any~$\nu$ feasible
in~\eqref{eq:mot-phi-refor-1} there is~$\mu$ feasible in~\eqref{eq:phi-1} with the same objective function value. To this end, choose any~$\nu$ feasible in~\eqref{eq:mot-phi-refor-1}, and define
\[
    \mu = (1-\epsilon)\cdot \frac{1}{n}\sum_{i=1}^n \EE_{\nu} [W|Z=\hat z_i]\cdot \delta_{\hat z_i} +\epsilon \cdot \EE_{\nu} [W|Z=\hat z_{n+1}] \cdot \delta_{\hat z_{n+1}}. 
\]
By construction, the discrete distribution $\mu$ belongs to $\mc P(\hat{\mc Z})$. In addition, it is absolutely continuous with respect to~$\rho$, and $\frac{\diff \mu}{\diff \rho}(Z)=\EE_{\nu} [W|Z]$. 
%
This implies that
\begin{align*}
    \E_{\rho}\!\left[\! \frac{\diff \hat \mu}{\diff \rho}(Z) \cdot \phi\left( \frac{\frac{\diff \mu}{\diff \rho}(Z)}{\frac{\diff \hat \mu}{\diff \rho}(Z)} \right)  \right] & = \E_{\nu}\!\left[\! \frac{\diff \hat \mu}{\diff \rho}(Z) \cdot \phi\left( \frac{\EE_{\nu} [W|Z]}{\frac{\diff \hat \mu}{\diff \rho}(Z)} \right)  \right] \\
    & \leq \E_{\nu}\!\left[\! \frac{\diff \hat \mu}{\diff \rho}(Z) \cdot \phi\left( \frac{W}{\frac{\diff \hat \mu}{\diff \rho}(Z)} \right)  \right] \leq r,
\end{align*}
where the equality exploits the feasibility of $\nu$ in~\eqref{eq:mot-phi-refor-1}, which ensures that the marginal distribution of~$Z$ under~$\nu$ coincides with~$\rho$, the first inequality follows from  Jensen's inequality and the tower property of conditional expectations, and the second inequality holds because $\nu$ satisfies the divergence constraint in~\eqref{eq:mot-phi-refor-1}. Hence, $\mu$ satisfies also the divergence constraint in~\eqref{eq:phi-1}. Similarly, we have
\[
    \EE_\mu[\ell(Z)] = \EE_\rho\!\left[\ell(Z)\cdot  \frac{\diff  \mu}{\diff \rho}(Z)\right] = \EE_{ \nu}\left[\ell(Z) \cdot \EE_{\nu} [W|Z]  \right] = \EE_{ \nu}[\ell (Z) \cdot W],
\]
where the first equality follows from the Radon-Nikodym theorem, the second equality exploits the definition of~$\mu$ as well as the relation~${\nu}_Z = \rho$, and the last equality follows from the tower property of conditional expectations. Thus, the objective value of~$\mu$ in~\eqref{eq:phi-1} matches that of~$\nu$ in~\eqref{eq:mot-phi-refor-1}. As $\nu$ was chosen arbitrarily, this implies that the optimal value of~\eqref{eq:phi-1} provides an upper bound on that of~\eqref{eq:mot-phi-refor-1}. In summary, we have thus shown that problems~\eqref{eq:phi-1} and~\eqref{eq:mot-phi-refor-1} share the same optimal value.

\noindent \textit{Step 2.}
We now recast~\eqref{eq:mot-phi-refor-1} as a problem over couplings $\pi\in\mc P((\mc V \times \mc W)^2)$.
\begin{equation}
    \begin{array}{ccll}
        & \max  & \EE_{\pi}\left[f(V,W)\right]\\
         &\suchthat & \pi \in \mc P((\mc V \times \mc W)^2)\\[-0.7ex]
        && \pi_{(\hat V, \hat W)} = \hat \nu \\ 
        && \EE_{ \pi}[W ] = 1\\
        &&  \EE_{\pi}[c((V,W),(\hat V,\hat W)) ] \leq r
    \end{array}
\label{eq:mot-phi-refor-2}
\end{equation}  
In analogy to {Step 1}, we first demonstrate that for any~$\nu$ feasible in~\eqref{eq:mot-phi-refor-1} there exists a~$\pi$ feasible in~\eqref{eq:mot-phi-refor-2} with the same objective value. In the following we interpret~$\hat \nu$ as a distribution of~$(\hat V, \hat W)$ on~$\mc V\times\mc W$. Select now any~$\nu$ feasible in~\eqref{eq:mot-phi-refor-1}, and construct~$\pi$ as follows. The marginal distribution of $(V, W)$ under~$\pi$ is set to~$\nu$, the conditional distribution of $\hat V$ given~$(V,W)$ under~$\pi$ is set to~$\delta_{V}$, and the conditional distribution of~$\hat W$ given $(V, W, \hat V)$ is set to $\hat{\nu}_{\hat W|\hat V}$. Clearly, we have $\pi \in \mc P((\mc V \times \mc W)^2)$. Next, select arbitrary Borel sets $A\subseteq\mc V$ and $B\subseteq\mc W$, and observe that
\begin{align*}
       \EE_\pi[\mathbf{1}_A(\hat V) \cdot \mathbf{1}_B(\hat W)] & =\int_{\mc V\times\mc W}\int_{A} \hat\nu_{\hat W|\hat V}(B|\hat v)\, \diff\delta_v(\hat v)\,\diff\nu(v,w)\\
       &= \int_{A\times\mc W}\hat\nu_{\hat W|\hat V}(B|\hat v)\, \,\diff\nu(\hat v,w)\\
       &= \int_{A}\hat\nu_{\hat W|\hat V}(B|\hat v)\, \,\diff\rho (\hat v) = \hat\nu(A\times B).
\end{align*}
Here, the first and the third equalities follow from the construction of~$\pi$ and the marginalization constraint~$\nu_Z=\rho$ in~\eqref{eq:mot-phi-refor-1}, respectively. The fourth equality exploits the definitions of~$\rho$ and~$\hat\nu$. As~$A$ and~$B$ were chosen freely, we thus have~$\pi_{(\hat V, \hat W)}=\hat\nu$. As $\pi_{(V, W)}=\nu$ and~$\nu$ is feasible in~\eqref{eq:mot-phi-refor-1}, we also have $\EE_{ \pi}[W ] = \EE_{\nu}[W] = 1$. Similarly,
\begin{align*}
    \EE_{\pi}[c((V,W),(\hat V,\hat W)) ] = \EE_{\pi}\left[ g(V) \cdot \phi\left(\frac{W}{g(V)}\right) \right]\leq r,
\end{align*}
where the equality follows from the definition of~$c$ and the observation that~$V=\hat V\in\hat{\mc Z}$ $\pi$-almost surely, while the inequality holds because $\tfrac{\diff\hat\mu}{\diff \rho}(V)=g(V)$ on the support of~$\rho$ and because~$\nu$ is feasible in~\eqref{eq:mot-phi-refor-1}. The objective function values of~$\nu$ and~$\pi$ in the respective problems are trivially equal by the construction of~$f$ and~$\pi$. As~$\nu$ was chosen arbitrarily, the optimal value of~\eqref{eq:mot-phi-refor-2} thus upper bounds that of~\eqref{eq:mot-phi-refor-1}.

Next, we show that any~$\pi$ feasible in~\eqref{eq:mot-phi-refor-2} corresponds to a~$\nu$ feasible in~\eqref{eq:mot-phi-refor-1} with the same objective function value. To this end, select any~$\pi$ feasible in~\eqref{eq:mot-phi-refor-2}, and define~$\nu$, which is viewed as a distribution of $(Z,W)$ in~\eqref{eq:mot-phi-refor-1}, as $\nu = \pi_{(\hat V,W)}$. Thus, we identify~$\nu$ with the distribution of $(\hat V,W)$ under~$\pi$. By construction, we thus have 
\[
    \nu_Z = \pi_{\hat V} = \hat \nu_{\hat V} = \rho,
\]
where the second equality holds because $\pi$ satisfies the marginalization constraint in~\eqref{eq:mot-phi-refor-2}, and the third equality exploits the definition of~$\hat\nu$. Thus, $\nu \in \mc P(\hat{\mc Z} \times \mc W)$ satisfies the marginalization constraint in~\eqref{eq:mot-phi-refor-1}. In addition, we have $\EE_{\nu} [W] = \EE_\pi[W] = 1$, where the second equality holds because~$\pi$ satisfies the normalization constraint in~\eqref{eq:mot-phi-refor-2}. Thus, $\nu$ satisfies the normalization constraint in~\eqref{eq:mot-phi-refor-1}. Similarly, we find
\begin{align*}
    \EE_{\nu}\left[ \frac{\diff \hat \mu}{\diff \rho}(Z)\cdot \phi\left( \frac{W}{\frac{\diff  \hat \mu}{\diff \rho}(Z)}\right) \right] & = \EE_{\nu}\left[ g(Z)\cdot \phi\left( \frac{W}{g(Z)}\right) \right]= \EE_\pi \left[ g(\hat V)\cdot \phi\left( \frac{W}{g(\hat V)}\right) \right] \\
    & = \EE_{\pi}[c((V,W),(\hat V,\hat W)) ] \leq r,
\end{align*}
where the first equality holds because $\frac{\diff \hat \mu}{\diff \rho}(z)$ is a measurable function of~$z$ that evaluates to $1/(1-\epsilon)$ if $z \in \{\hat{z}_i\}_{i=1}^n$. Thus, it coincides with~$g(z)$ almost surely with respect to~$\nu_Z=\rho$. The second equality holds because the distribution~$\nu$ of $(Z,W)$ coincides with the distribution of $(\hat V,W)$ under~$\pi$, and the third equality exploits the definition of~$c$ and the feasibility of~$\pi$ in~\eqref{eq:mot-phi-refor-2}, which guarantees that $\EE_{\pi}[c((V,W),(\hat V,\hat W)) ]$ is finite and thus implies that $V= \hat V$ $\pi$-almost surely. The inequality, finally, holds because $\pi$ satisfies the divergence constraint in~\eqref{eq:mot-phi-refor-2}. Hence, $\nu$ is feasible in~\eqref{eq:mot-phi-refor-1}. As the distribution~$\nu$ of $(Z,W)$ matches the distribution of $(\hat V,W)$ under~$\pi$, we also have
\[
    \EE_{\nu}[\ell(Z)\cdot W ]=\EE_{\pi}[f(\hat V,W) ]=\EE_{\pi}[f(V,W) ],
\]
where the second equality holds again because $V= \hat V$ $\pi$-almost surely. Thus, the objective value of~$\nu$ in~\eqref{eq:mot-phi-refor-1} matches that of~$\pi$ in~\eqref{eq:mot-phi-refor-2}. The optimal value of~\eqref{eq:mot-phi-refor-1} thus provides an upper bound on that of~\eqref{eq:mot-phi-refor-2}. In summary, we have shown that problems~\eqref{eq:mot-phi-refor-1} and~\eqref{eq:mot-phi-refor-2} share the same optimal value. This completes the proof.
\end{proof}

\begin{rmk}[Properties of $c$]
The cost function $c$ defined in~\eqref{eq:cost-phi-mot-equiv} satisfies
\begin{equation*}
\begin{aligned}
c((v,w),(\hat v, \hat w)) & = 
\begin{cases}
    \frac{1}{1-\epsilon} \cdot \phi((1-\epsilon) \cdot w) & \text{if}~\hat v  \in \{\hat{z}_i\}_{i=1}^{n}~\text{and}~v =\hat v, \\ 
    \phi'_\infty \cdot w & \text{if}~\hat v =\hat{z}_{n+1}~\text{and}~v= \hat v,\\ 
    +\infty & \text{otherwise}. 
\end{cases}
\end{aligned}
\end{equation*}
This representation reveals that $c((v,w),(\hat v, \hat w))$ is  lower semicontinuous and bounded below.  It also shows that $c((v,w),(\hat v, \hat w))$ vanishes if $(v, w) = (\hat v, \hat w)$ and if $(v,w)=(\hat z_i, (1-\epsilon)^{-1})$ for some $i\in [n]$ or $(v,w)=(\hat z_{n+1}, 0)$. These properties will be instrumental for the applicability of the strong duality theorem derived in Section~\ref{sec:duality}.
\end{rmk}

\subsection{Sinkhorn Discrepancy-Based DRO}
The Sinkhorn discrepancy between two distributions is defined as the optimal value of an entropy-regularized OT problem.
\begin{definition}[Sinkhorn discrepancy]
If $\mc Z \subseteq \R^{d_z}$ is a convex and closed set, $d:\mathcal{ Z \times Z}\rightarrow\R_+$ is a lower semicontinuous cost function satisfying the identity of indiscernibles~($d(z, \hat z) = 0$ if and only if $z =\hat z$), $\eta, \hat \eta \in \mc P(\mc Z)$ are arbitrary reference distributions and~$\epsilon \geq 0$ is a regularization parameter, then the Sinkhorn discrepancy induced by~$d$, $\eta$, $\hat \eta$ and~$\epsilon$ is the function $\mathds W_d^\epsilon: \mc P(\mc Z) \times \mc P(\mc Z)\rightarrow [0,+\infty]$ defined through
\[
    \mathds W_d^\epsilon(\mu, \hat \mu) = \inf_\pi \Big\{\EE_{\pi}[d(Z, \hat Z)+\varepsilon \mathds D_{\textrm{\em KL}}(\pi, \eta \otimes \hat \eta)] : \pi \in \mc P(\mc Z \times {\mc Z}), \,\pi_Z = \mu, \,\pi_{\hat Z} = \hat \mu\Big\},
\]
where 
$\mathds D_{\textrm{\em KL}}$ stands for the Kullback-Leibler divergence (see Definition~\ref{defi-phi} and Table~\ref{tab:phi-divergence}).
\label{def:sinkhorn}
\end{definition}

If $\epsilon=0$, then the Sinkhorn discrepancy collapses to the classical OT discrepancy induced by~$d$. If additionally $d(z,\hat z)=\|z-\hat z\|^p$ for some $p\in\mathbb N$, then the Sinkhorn discrepancy collapses to the $p$-th power of the $p$-Wasserstein metric \citep{ref:villani2009optimal}.


Next, we will study the Sinkhorn discrepancy-based DRO problem
\begin{equation}
    \sup\limits_{\mu \in \mc P(\mc Z)} \left\{ \EE_{\mu}[\ell(Z)] : \mathds W_d^\varepsilon(\mu, \hat \mu) \leq r  \right\},
    \label{eq:sinkhorn-dro}
\end{equation}
and we will show that~\eqref{eq:sinkhorn-dro} can be recast as an instance of the OT-DRO problem~\eqref{eq:mot-dro} with conditional moment constraints. Note that problem~\eqref{eq:sinkhorn-dro} becomes infeasible for small (positive) values of~$r$. The feasibility of~\eqref{eq:sinkhorn-dro} also depends on the choice of the reference distributions $\eta$ and $\hat\eta$. One can show that, for any fixed~$\mu$, $\hat\mu$ and~$\eta$, $\mathds W_d^\varepsilon(\mu, \hat \mu)$ is smallest if $\hat \eta = \hat \mu$. In order to minimize the radius of the ambiguity set below which~\eqref{eq:sinkhorn-dro} becomes infeasible, we will henceforth assume that~$\hat \eta = \hat \mu$. Specifically, in this case one can show  that~\eqref{eq:sinkhorn-dro} is feasible if and only if 
\[
    r\geq \underline r = - \varepsilon \cdot \EE_{\hat \mu} \left[ \log \left( \EE_{\eta}\left[ \exp\left(-d(Z,\hat Z)/\varepsilon \right) \right]\right) \right ];
\]
see~\cite{wang2021sinkhorn}. Note that the infeasibility threshold $\underline r$ on the right hand side of the above inequality is non-decreasing in $d$ and, as $d\geq 0$, thus nonnegative. The following assumption ensures that the Sinkhorn DRO problem~\eqref{eq:sinkhorn-dro} is well-defined.

\begin{assumption}[Assumption~1 in \cite{wang2021sinkhorn}]
\label{ass:sinkhorn-dro}
The following hold. 
\begin{enumerate}[label=\normalfont(\alph*)]
\item  We have $\eta(d(Z,\hat z)<\infty) = 1$ for $\hat\mu$-almost every $\hat z\in\mc Z$. 
\item We have $\EE_{ \eta}[ \exp(-d(Z,\hat z)/\varepsilon )]<\infty$ for $\hat\mu$-almost every $\hat z\in\mc Z$. 
\end{enumerate}
\end{assumption}

The following theorem shows that if $\mathds D=\mathds W_d^\epsilon$ is a Sinkhorn discrepancy, then the worst-case expectation problem in~\eqref{eq:dro} can be recast as an instance of problem~\eqref{eq:mot-dro}.

\begin{theorem}[Sinkhorn discrepancy-based DRO]
Let $(\ell, \mc Z, \hat \mu, \mathds D, r)$ be an instance of the worst-case expectation problem in~\eqref{eq:dro}, where $\mathds D=\mathds W_d^\epsilon$ is the Sinkhorn discrepancy induced by~$d$, $\eta$, $\hat \eta=\hat\mu$ and~$\epsilon>0$. Suppose that \Cref{ass:sinkhorn-dro} holds and $r\geq\underline r$. Set $\mathcal{V} = \mathcal{V}_1\times \mathcal{V}_2$ with $\mathcal{V}_1=\mathcal{V}_2=\mc Z$, $\mathcal{W}=\mathbb{R}_+$, ${\mc G} = \sigma(\hat V_2)$,  $f(v, w) =\ell(v_1) \cdot w$~and
\begin{equation}
    c((v, w), (\hat v, \hat w)) = \left\{ \begin{array}{ll}
        \varepsilon \cdot \max\{ \phi_{\textrm{\em KL}}(w)-\phi_{\textrm{\em KL}}(\hat w),0\} & \text{if $v=\hat v$}, \\
        \infty & \text{if $v\neq \hat v$,}
    \end{array}\right.
\end{equation}
where $\phi_{\textrm{\em KL}}$ denotes the entropy function of the Kullback-Leibler divergence; see Table~\ref{tab:phi-divergence}. In addition, set the nominal distribution of $(\hat V, \hat W)$ to $\hat{\nu} = \hat{\gamma} \otimes \delta_1$, where the distribution~$\hat \gamma$ of $\hat V=(\hat V_1,\hat V_2)$ is constructed as follows. The marginal distribution of~$\hat V_2$ under~$\hat\gamma$ is given by~$\hat \gamma_{\hat V_2}= \hat\mu$, and the distribution of~$\hat V_1$ conditional on~$\hat V_2=\hat v_2$ is absolutely continuous with respect to~$\eta$ with Radon-Nikodym derivative
\begin{equation*}
    \frac{\diff \hat \gamma_{\hat V_1|\hat V_2=\hat v_2}}{\diff \eta} (\hat v_1) = \frac{ \exp\left(-d(\hat v_1,\hat v_2)/\varepsilon \right)}{\EE_{\eta} \left[\exp(-d(\hat V_1,\hat v_2)/\varepsilon )\right]} . 
\end{equation*}
Then, the inner worst-case expectation problem in~\eqref{eq:dro} has the same optimal value as the instance $(f, \mc V, \mc W, \mc G, \hat \nu, c, r-\underline r)$ of problem~\eqref{eq:mot-dro}. 
\label{theorem:sinkhorn-mot}
\end{theorem}

\begin{proof}
As $\mathds D=\mathds W_d^\epsilon$, problem~\eqref{eq:dro} is equivalent to~\eqref{eq:sinkhorn-dro}. By \cite[Lemma~2]{wang2021sinkhorn}, which applies thanks to \Cref{ass:sinkhorn-dro}, and as $\hat\gamma_{\hat Z}=\hat \mu$, we can then reformulate~\eqref{eq:sinkhorn-dro} as
\begin{equation}
    \begin{array}{ccll}
         & \sup  & \EE_{ \gamma}[\ell(Z)]\\
         &\suchthat & \gamma \in \mc P(\mc Z \times \mc Z),~\gamma\ll\hat\gamma\\ [-0.5ex]
         && {\gamma}_{\hat Z} = \hat \mu\\
         && \varepsilon \cdot \EE_{\gamma}\left[ \log\left(\frac{\diff \gamma}{\diff \hat\gamma }(Z,\hat{Z}) \right) \right] \leq r-\underline r,
    \end{array}
\label{eq:mot-sinkhorn-refor-1}
\end{equation} 
where $\underline r$ denotes the infeasibility threshold of problem~\eqref{eq:sinkhorn-dro}. 
The rest of the proof parallels that of \Cref{theorem:phi-div-mot}. By introducing a random variable~$W$ for the likelihood ratio $\frac{\diff \gamma}{\diff \hat\gamma }(Z,\hat{Z})$ we will first recast~\eqref{eq:mot-sinkhorn-refor-1} as an optimization problem over probability distributions in $\mc P(\mc Z \times \mc Z \times \mc W)$ (Step~1). Next, we will lift this reformulation to an optimization problem over couplings in $\mc P((\mc V \times \mc W)^2)$ (Step~2).

\noindent\textit{Step~1.} We first recast problem~\eqref{eq:mot-sinkhorn-refor-1} as an optimization problem over all probability distributions $\bar{\gamma} \in \mathcal{P}(\mc Z \times \mc Z\times \mc W)$, where $\mathcal{W} = \mathbb{R}_+$ is the space of likelihood ratios.
\begin{equation}
    \begin{array}{ccll}
         & \sup  & \EE_{\bar \gamma}\left[\ell(Z)\cdot W\right]\\
         &\suchthat & \bar \gamma \in \mc P(\mc Z \times \mc Z \times W)\\[-0.5ex]
         && {\bar \gamma}_{(Z,\hat Z)} = \hat \gamma  \\ 
         && \EE_{\bar \gamma} [W|\hat Z] = 1 \quad \hat \mu \text{-a.s.}\\
        && \varepsilon  \cdot \EE_{ \bar \gamma} \left[\phi_{\textrm{KL}}(W) \right] \leq r-\underline r. 
    \end{array}
\label{eq:mot-sinkhorn-refor-2}
\end{equation} 
We will first show that any~$\gamma$ feasible in~\eqref{eq:mot-sinkhorn-refor-1} gives rise to a $\bar \gamma$ feasible in~\eqref{eq:mot-sinkhorn-refor-2} with the same objective value. Thus, choose any distribution~$\gamma$ of $(Z, \hat Z)$ feasible in~\eqref{eq:mot-sinkhorn-refor-1}, and define the distribution $\bar\gamma$ of $(Z, \hat Z, W)$ as follows. The marginal distribution of $(Z,\hat Z)$ under~$\bar\gamma$ is set to $\hat\gamma$, and the conditional distribution of~$W$ given~$(Z,\hat Z)$ under~$\bar\gamma$ is set to the Dirac distribution at~$\tfrac{\diff \gamma}{\diff \hat\gamma }(Z,\hat{Z})$. By construction, $\bar\gamma$ belongs to $\mc P(\mc Z \times \mc Z \times W)$ and satisfies the marginalization constraint in~\eqref{eq:mot-sinkhorn-refor-2}. We further have
\begin{equation*}
    \label{eq:sinkhorn-equ-1}
    \begin{aligned}
    \EE_{\bar \gamma} [W |\hat Z ] & = \EE_{ \bar \gamma }\left[\left. \EE_{ \bar \gamma }[ W  |Z,\hat Z ] \right| \hat Z\right]= 
    \EE_{ \hat \gamma }\left[\left.\tfrac{\diff \gamma}{\diff \hat\gamma }(Z,\hat{Z})\right| \hat Z\right] \\
    &= \EE_{ \hat \gamma }\left[\left. \tfrac{\diff \gamma_{Z|\hat{Z}}}{\diff \hat \gamma_{Z|\hat{Z}}}(Z|\hat Z) \right| \hat Z\right] =1 \quad \hat \gamma_{\hat Z}\text{-a.s}.,
    \end{aligned}
\end{equation*}
where the second equality follows from the construction of $\bar\gamma$, the third equality holds because the marginal distribution of~$\hat Z$ is given by~$\hat\mu$ both under~$\gamma$ and~$\hat \gamma$, and the fourth equality follows from the Radon-Nikodym theorem. Thus, $\bar\gamma$ satisfies the normalization constraint in~\eqref{eq:mot-sinkhorn-refor-2}. Using a similar reasoning, we further obtain
\begin{align*}
    \varepsilon\cdot \EE_{\bar \gamma} [\phi_{\textrm{KL}}(W)] & =\varepsilon\cdot \EE_{ \bar \gamma }\left[ \EE_{ \bar \gamma }[ \phi_{\textrm{KL}}(W)  |Z,\hat Z ] \right] = \varepsilon\cdot \EE_{ \hat \gamma }\left[ \phi_{\textrm{KL}}\left(\tfrac{\diff \gamma}{\diff \hat \gamma}(Z,\hat Z)\right) \right] \\
    & = \varepsilon\cdot \EE_{\gamma}\left[ \log\left(\tfrac{\diff \gamma}{\diff \hat\gamma}(Z,\hat{Z})\right) \right]\leq r-\underline r,
\end{align*}
where the second equality follows from the construction of~$\bar \gamma$, while the third equality exploits the definition of $\phi_{\textrm{KL}}(t)=t \cdot \log(t) -t +1$ and the Radon-Nikodym theorem, and the inequality holds because~$\gamma$ satisfies the last constraint in~\eqref{eq:mot-sinkhorn-refor-1}. Hence, $\bar \gamma$ satisfies the last constraint in~\eqref{eq:mot-sinkhorn-refor-2}. Similarly, we find
\begin{align*}
    \EE_{\bar \gamma}\left[\ell(Z)\cdot W\right] =\EE_{ \bar \gamma }\left[\ell(Z) \cdot \EE_{ \bar \gamma }[ W |Z,\hat Z ]\right] = \EE_{ \hat \gamma }\left[ \ell(Z) \cdot\tfrac{\diff \gamma}{\diff \hat\gamma }(Z,\hat{Z})\right] 
    = \EE_{\gamma}[\ell(Z)],
\end{align*}
where the last equality exploits again the observation that~$\hat Z$ has the same marginal distribution under~$\gamma$ and~$\hat \gamma$. Hence, the objective value of~$\gamma$ in~\eqref{eq:mot-sinkhorn-refor-1} matches that of~$\bar\gamma$ in~\eqref{eq:mot-sinkhorn-refor-2}. The optimal value of~\eqref{eq:mot-sinkhorn-refor-2} thus upper bounds that of~\eqref{eq:mot-sinkhorn-refor-1}.

Next, we show that for any~$\bar \gamma$ feasible in~\eqref{eq:mot-sinkhorn-refor-2} there is~$\gamma$  feasible in~\eqref{eq:mot-sinkhorn-refor-1} with the same objective function value. To see this, select any~$\bar \gamma$ feasible in~\eqref{eq:mot-sinkhorn-refor-2}, and construct~$\gamma$ as the distribution that is absolutely continuous with respect to~$\hat \gamma$ with Radon-Nikodym derivative $\frac{\diff\gamma}{\diff\hat\gamma}(Z,\hat Z)=\EE _{\bar \gamma}[W|Z,\hat{Z}]$. One readily verifies that
\[
    \EE_{\hat\gamma} \left[\tfrac{\diff\gamma}{\diff\hat\gamma}(Z,\hat Z)\right] = \EE_{\bar\gamma} \left[\tfrac{\diff\gamma}{\diff\hat\gamma}(Z,\hat Z)\right] = \EE_{\bar\gamma} \left[\EE _{\bar \gamma}[W|Z,\hat{Z}]\right] =1,
\]
where the first equality holds because $\bar\gamma_{(Z,\hat Z)}=\hat\gamma$ thanks to the marginalization constraint in~\eqref{eq:mot-sinkhorn-refor-2}, and the third equality exploits the normalization constraint in~\eqref{eq:mot-sinkhorn-refor-2}. Thus, $\gamma$ belongs to~$\mc P(\mc Z \times \mc Z)$. Similarly, for any Borel set $B\subseteq \mc Z$ we have
\[
    \gamma_{\hat Z}(B) = \EE_\gamma[\mathbf{1}_B(\hat Z)] = \EE_{\hat \gamma}[\mathbf{1}_B(\hat Z)\cdot \EE_{\bar \gamma}[W|Z,\hat Z]] = \EE_{\bar \gamma}[\mathbf{1}_B(\hat Z)\cdot W]= \EE_{\bar \gamma}[\mathbf{1}_B(\hat Z)] =\hat \mu (B).
\]
Here, the first and the fourth equalities exploit again the marginalization and normalization constraints in~\eqref{eq:mot-sinkhorn-refor-2}, respectively. We thus find ${\gamma}_{\hat Z} = \hat \mu$.
Next, we have
\begin{align*}
    & \varepsilon \cdot  \EE_{\gamma}\left[ \log\left(\tfrac{\diff \gamma(Z,\hat{Z}) }{\diff \hat\gamma (Z,\hat{Z})}\right) \right] = \varepsilon \cdot  \EE_{\hat \gamma}\left[ \EE_{\bar\gamma}[W | Z , \hat Z]\cdot \log\left(\tfrac{\diff \gamma(Z,\hat{Z}) }{\diff \hat\gamma (Z,\hat{Z})}\right) \right] \\ = \;& \varepsilon \cdot  \EE_{\bar \gamma}\left[ \EE_{\bar\gamma}[W | Z , \hat Z]\cdot \log\left(\EE_{\bar\gamma}[W | Z , \hat Z] \right) \right] =  \varepsilon \cdot \EE_{\bar \gamma}\left[ \phi_{\textrm{KL}}\left(\EE_{\bar \gamma}[W|Z,\hat{Z}] \right)\right] & \\ 
    \leq \; & \varepsilon\cdot  \EE_{\bar \gamma}\left[ \phi_{\textrm{KL}}\left(W \right)\right] \leq r-\underline r,
\end{align*}
where the first equality follows from the Radon-Nikodym theorem, the second equality holds because $\bar\gamma_{(Z,\hat Z)}=\hat\gamma$ and $\frac{\diff\gamma}{\diff\hat\gamma}(Z,\hat Z)=\EE _{\bar \gamma}[W|Z,\hat{Z}]$, and the third equality exploits the definition of~$\phi_{\textrm{KL}}$ and the observation that $\EE_{\bar \gamma}[W]=1$. The two inequalities then follow from Jensen's inequality and the last constraint in~\eqref{eq:mot-sinkhorn-refor-2}, respectively. Thus, $\gamma$ satisfies the last constraint in~\eqref{eq:mot-sinkhorn-refor-1}. Using a similar reasoning, we finally find
\begin{align*}
    \EE_{\gamma}[\ell(Z)] &= \EE_{\hat \gamma}\left[ \ell(Z)\cdot \EE_{\bar \gamma}[W | Z , \hat Z] \right]= \EE_{\bar \gamma}\left[ \ell(Z)\cdot \EE_{\bar \gamma}[W | Z , \hat Z] \right] = \EE_{\bar \gamma}[\ell (Z) \cdot W].
\end{align*}
Hence, the objective value of~$\gamma$ in~\eqref{eq:mot-sinkhorn-refor-1} matches that of~$\bar\gamma$ in~\eqref{eq:mot-sinkhorn-refor-1}. This implies that the optimal value of~\eqref{eq:mot-sinkhorn-refor-1} provides an upper bound on that of~\eqref{eq:mot-sinkhorn-refor-2}. In summary, we have thus shown that problems~\eqref{eq:mot-sinkhorn-refor-1} and~\eqref{eq:mot-sinkhorn-refor-2} share the same optimal value.

\noindent \textit{Step~2.} We now recast~\eqref{eq:mot-sinkhorn-refor-2} as a problem over couplings $\pi \in\mc P((\mc V \times \mc W)^2)$.
\begin{equation}
    \begin{array}{ccll}
        & \max  & \EE_{\pi}\left[f(V,W)\right]\\
         &\suchthat & \pi \in \mc P((\mc V \times \mc W)^2)\\[-0.7ex]
        && \pi_{(\hat V, \hat W)} = \hat \nu \\ 
        && \EE_{ \pi}[W|\hat V_2 ] = 1\\
        &&  \EE_{\pi}[c((V,W),(\hat V,\hat W)) ] \leq r-\underline r.
    \end{array}
\label{eq:mot-sinkhorn-refor-3}
\end{equation}
To show the equivalence of~\eqref{eq:mot-sinkhorn-refor-2} and~\eqref{eq:mot-sinkhorn-refor-3}, we use essentially the same reasoning as in Step~2 in the proof of Theorem \ref{theorem:phi-div-mot}. Details are omitted to avoid redundancy. 
\end{proof}
Note that for $\varepsilon=0$, problem \eqref{eq:sinkhorn-dro} simplifies to a standard OT-DRO problem {\em without} conditional moment constraints. This problem is trivially an instance of~\eqref{eq:mot-dro}.

\begin{theorem}[OT discrepancy-based DRO]
\label{theorem:wass_mot}
Let $(\ell, \mc Z, \hat \mu, \mathds D, r)$ be an instance of the worst-case expectation problem in~\eqref{eq:dro}, where $\mathds D=\mathds W_d^0$ is the OT discrepancy induced by~$d$. Define $\mathcal{V} = \mc Z$, $\mathcal{W}=\mathbb{R}_+$, ${\mc G} = \{\emptyset,(\mc V\times\mc W)^2 \}$,  $f(v, w) =\ell(v) \cdot w$ and 
\[
 c((v,w),(\hat v,\hat w))=\left\{ \begin{array}{ll}
      d(v,\hat v) & \text{if }w=\hat w, \\
      \infty & \text{if }w\neq \hat w.
 \end{array} \right.
\]
In addition, set the nominal distribution of $(\hat V, \hat W)$ to $\hat{\nu} = \hat{\mu} \otimes \delta_1$. Then, the worst-case expectation problem in~\eqref{eq:dro} has the same optimal value as~\eqref{eq:mot-dro}.
\end{theorem}

\begin{proof}
Problem~\eqref{eq:dro} is equivalent to~\eqref{eq:sinkhorn-dro} with $\varepsilon =0$ and can be reformulated as
\begin{equation}
\begin{array}{ccll}
    & \sup & \mathbb{E}_{\gamma}[\ell(Z)] \\
    &\suchthat &\gamma \in \mathcal{P}(\mathcal{Z}\times \mathcal{Z})\\[-0.5ex]
    && \gamma_{\hat{Z}} = \hat{\mu} \\
    &&\mathbb{E}_{\gamma}[d(Z,\hat{Z})]\leq r
\end{array}
\label{eq:mot-ot-refor-2}
\end{equation}
by using~\cite[Theorem~4.1]{ref:villani2009optimal}. The rest of the proof is elementary and thus omitted.
\end{proof}

\section{Strong Duality}
\label{sec:duality}
We will now show that if~$\mc G = \sigma(\hat V)$, then the OT-DRO problem~\eqref{eq:mot-dro-extended} with conditional moment constraints admits the following strong dual.
\begin{equation}
\label{eq:dual-mot}
  \inf\limits_{\substack{\lambda \in \R_+ \\ \psi\in \mc  C_b(\mc V)}}  \lambda r +
\EE_{\hat \nu}\left [\sup\limits_{(v,w) \in \mc V \times \mc W } f(v, w) + \psi(\hat V)\cdot (w-1) -\lambda c((v, w), (\hat V, \hat W))\right ]
\end{equation}
Here, $\mc C_b (\mc V)$ denotes the family of all bounded continuous functions on~$\mc V$. This duality result will enable us to solve interesting instances of the proposed OT-DRO problem~\eqref{eq:mot-dro} with conditional moment constraints. In addition, the dual problem~\eqref{eq:dual-mot} will provide insights into the structure of the worst-case distribution that solves the primal problem~\eqref{eq:mot-dro-extended}. Our duality result relies on the following  assumption.
\begin{assumption}[Conditions for strong duality]
\label{ass:strong_duality}
\begin{enumerate}[label=\normalfont(\alph*)]
\item The cost function $c: (\mc V \times \mc W)^2 \rightarrow (-\infty, +\infty]$ is lower semicontinuous and satisfies $c((v,w),(\hat v,\hat w)) =0$ whenever $(v,w) = (\hat v,\hat w)\in \text{\rm supp}(\hat\nu)$. 
\item The loss function $f:\mc V\times\mc W\rightarrow [-\infty, +\infty)$ is upper semicontinuous and satisfies $\EE_{\hat \nu}[|f(\hat V, \hat W)|]<\infty$. 
\item The support sets $\mc V\subseteq \R^{d_v}$ and $\mc W\subseteq \R_+$ are compact.
\end{enumerate}
\end{assumption}

\begin{theorem}[Strong duality]
\label{thm:strong-duality}
If \Cref{ass:strong_duality} holds, then the supremum of the primal problem~\eqref{eq:mot-dro-extended} equals the infimum of the dual problem~\eqref{eq:dual-mot}.
\end{theorem}
\begin{proof}
Define the Lagrangian function
\begin{equation*}
    L(\pi;\lambda,\psi) = \lambda r + \EE_{\pi}[f(V, W)+\psi(\hat V)\cdot(W-1)-\lambda c((V, W), (\hat V, \hat W))],
\end{equation*}
where the Lagrange multipliers~$\lambda \in \R_+$ and $\psi \in \mc C_b(\mc V)$ are associated with the mass transportation and the moment constraints in problem~\eqref{eq:mot-dro-extended}, respectively, while the primal decision variable~$\pi$ ranges over the feasible set 
\[
    \Pi_{\hat \nu}= \left\{\pi \in \mc P((\mc V\times \mc W)^2) \; :\; \pi_{(\hat V, \hat W)}  = \hat \nu \right\}.
\]
We will use Sion's minimax theorem to show that
\begin{equation}
    \label{eq:sion}
    \sup_{\pi\in\Pi_{\hat \nu}} \inf_{\lambda \in \R_+, \psi \in \mc C_b(\mc V) } L(\pi;\lambda,\psi) = \inf_{\lambda \in \R_+, \psi \in \mc C_b(\mc V)}\sup_{\pi\in\Pi_{\hat \nu}}L(\pi;\lambda,\psi). 
\end{equation}
Before that, however, we first demonstrate that the max-min problem on the left-hand side of~\eqref{eq:sion} is equivalent to the primal problem~\eqref{eq:mot-dro-extended} and that the min-max problem on the right-hand side of~\eqref{eq:sion} is equivalent to the dual problem~\eqref{eq:dual-mot}. To establish the equivalence of~\eqref{eq:mot-dro-extended} and the max-min problem in~\eqref{eq:sion}, note that
\begin{align*}
    L(\pi;\lambda,\psi)  = \EE_{\pi}[f(V, W)] +\EE_{\pi}[\psi(\hat V)\cdot(W-1)] + \lambda\left (  r- \EE_{\pi}[ c((V, W), (\hat V, \hat W))] \right). 
\end{align*}
In addition, recall that $\mc W$ is compact, which ensures that the mapping
\begin{align*}
    m(\psi)= \EE_{\pi}[\psi(\hat V)\cdot(W-1)] & = \EE_{\pi}[\psi(\hat V)\cdot \EE_{\pi}[W-1|\hat V]]
\end{align*}
constitutes a continuous linear functional on~$\mc C_b(\mc V)$ under the uniform topology. By~\cite[Corollary~14.15]{aliprantis2006infinite}, which applies because~$\mc V$ is compact, there is a one-to-one correspondence between the continuous linear functionals on~$\mc C_b(\mc V)$ and the signed Borel measures on~$\mc V$. That is, $m(\psi)=\EE_\rho[\psi(\hat V)]$ for a unique signed Borel measure~$\rho$, and
\[
    \inf_{\psi\in\mc C_b(\mc V)} m(\psi)=\left\{ \begin{array}{cl}
        0 & \text{if }\rho=0, \\
        -\infty & \text{if }\rho\neq 0.
    \end{array}\right.
\]
From the construction of~$m$ it is clear that~$\rho(B)=\EE_{\pi}[\mathbbm 1_B(\hat V)\cdot \EE_{\pi}[W-1|\hat V]]$ for every Borel set~$B\subseteq\mc V$. We thus conclude that~$\rho$ is the zero measure if and only if~$\EE_{\pi}[W-1|\hat V]=0$ $\pi$-almost surely. This insight readily implies that 
\begin{equation*}
    \inf_{\lambda \in \R_+, \psi \in \mc C_b(\mc V) } L(\pi;\lambda,\psi) = \left\{ \!\! \begin{array}{ll}
    \EE_{\pi}[f(V, W)] & \text{if }\left\{\begin{array}{l}\EE_\pi[W|\hat V]=1~\pi\text{-a.s.\ and} \\ \EE_{\pi}[ c((V, W), (\hat V, \hat W))]\leq r,\end{array} \right.\\
    -\infty & \text{otherwise.}
    \end{array}\right.
\end{equation*}
Thus, the left-hand side of~\eqref{eq:sion} is indeed equivalent to the primal problem~\eqref{eq:mot-dro-extended}.

Next, we show that the min-max problem on the right-hand side of~\eqref{eq:sion} is equivalent to the dual problem~\eqref{eq:dual-mot}. Clearly, the right-hand side of~\eqref{eq:sion} lower bounds~\eqref{eq:dual-mot}. To establish the reverse inequality, we fix any~$\lambda\in\R_+$ and~$\psi\in\mc C_b(\mc V)$ and construct~$\pi\in\Pi_{\hat \nu}$ such that $L(\pi;\lambda,\psi)$ equals the objective value of~$\lambda$ and~$\psi$ in~\eqref{eq:dual-mot}. This can be achieved by constructing a Borel measurable map $s:\mc V\times\mc W\rightarrow \mc V\times\mc W$ with
\[
    s(\hat v,\hat w)\;\in\;\mathop{\arg\max}_{(v,w)\in\mc V\times\mc W}f(v, w) + \psi(\hat v)\cdot (w-1) -\lambda c((v, w), (\hat v, \hat w)) \quad \forall (\hat v,\hat w)\in \mc V\times\mc W.
\]
Note that the above maximization problem has a compact feasible set, and its objective function is jointly upper semicontinuous in the decisions~$(v,w)$ and the parameters~$(\hat v,\hat w)$. Hence, the argmax multifunction is weakly measurable as well as non-empty and closed-valued. The map~$s$ thus exists by the Kuratowski-Ryll-Nardzewski selection theorem \citep[Theorem~18.13]{aliprantis2006infinite}. Next, we construct~$\pi\in\Pi_{\hat \nu}$ such that the conditional distribution of~$(V,W)$ given~$(\hat V, \hat W)$ equals the Dirac distribution $\delta_{s(\hat V,\hat W)}$. In this case, $L(\pi;\lambda,\psi)$ matches the objective value of~$\lambda$ and~$\psi$ in~\eqref{eq:dual-mot}, implying that the right-hand side of~\eqref{eq:sion} is indeed equivalent to the dual problem~\eqref{eq:dual-mot}.

We now use Sion's minimax theorem to prove~\eqref{eq:sion}. Note that~$\Pi_{\hat \nu}$ is both weakly compact and convex. Indeed, $\Pi_{\hat \nu}$ is weakly closed because it coincides with the pre-image of the weakly closed singleton set $\{\hat\nu\}$ under the mapping $\pi\mapsto \pi_{(\hat V, \hat W)}$, which is weakly continuous by~\cite[Theorem~15.14]{aliprantis2006infinite}. In addition, as~$\mc V$ and~$\mc W$ are compact, the distribution family $\mc P((\mc V\times \mc W)^2)$ as well as its subset~$\Pi_{\hat \nu}$ are tight. Prokhorov's theorem  \cite[Theorem~2.4]{van2000asymptotic} thus ensures that~$\Pi_{\hat \nu}$ is weakly compact. The convexity of~$\Pi_{\hat \nu}$ is easy to verify. Note also that~$\mc C_b(\mc V)$ is convex and that the Lagrangian~$L(\pi;\lambda,\psi)$ is bilinear in~$\pi$ and~$(\lambda,\psi)$. In addition, $L(\pi;\lambda,\psi)$ us weakly upper semicontinuous in~$\pi$ by \cite[Proposition~3.3]{Kuhn_Shafiee_Wiesemann_2025}, which applies because~$c$ is lower semicontinuous, $f$ is upper semicontinuous and $\psi$ is continuous, while~$\mc V$ and~$\mc W$ are compact.

For Sion's minimax theorem to apply, it remains to be shown that~$L(\pi;\lambda,\psi)$ is continuous in~$(\lambda,\psi)$ with respect to the uniform topology on~$\R_+ \times \mc C_b (\mc V)$. To show this, let $\{\lambda_n\}_{n\in\mathbb N}$ be any sequence in~$\mathbb R_+$ converging to~$\lambda$ in the usual sense, and let $\{\psi_n\}_{n\in\mathbb N}$ be any sequence in $\mc C_b(\mc V)$ converging to~$\psi$ with respect to the uniform topology. Thus, there exist~$\bar \lambda ,\bar \psi \in \R_+$ with $\sup_{n \to \infty} |\lambda_n|\leq \bar \lambda $ and  $\sup_{v\in \mc V} \|\psi_n(v)\| <\bar \psi$ for all $n\in\mathbb N$. The dominated convergence theorem then ensures that
\[
    \mathop{\lim}_{n\rightarrow +\infty} L(\pi;\lambda_n,\psi_n)= L(\pi;\lambda,\psi).
\]
This means that the Lagrangian $L(\pi; \lambda, \psi)$ is continuous in~$(\lambda, \psi)$ with respect to the uniform topology on the product space~$\R_+\times \mc C_b(\mc V)$.

Having checked all underlying technical conditions, we may now conclude that strong duality between~\eqref{eq:mot-dro-extended} and~\eqref{eq:dual-mot} follows indeed from Sion's minimax theorem.
\end{proof}

\begin{rmk}[Other $\sigma$-algebras]
    \label{rem:trivial-sigma-algebra}
    If $\mathcal G = \sigma(P\hat V)$ for some projection~$P$ on~$\R^{d_v}$, then $\mathcal C_b(\mathcal V)$ in the dual problem~\eqref{eq:dual-mot} must be replaced with the family of all bounded continuous functions of~$P\hat V$. In this case, one can show that \Cref{ass:strong_duality} still implies strong duality. If $P=0$, then this function class is isomorphic to~$\R$.
\end{rmk}

In conventional OT-DRO {\em without} conditional moment constraints, a compactness condition akin to \Cref{ass:strong_duality}\,(c) is not needed to establish strong duality; see, {\em e.g.},~\cite{zhang2022simple}. Unfortunately, we conjecture that \Cref{thm:strong-duality} ceases to hold if \Cref{ass:strong_duality}\,(c) is relaxed. Indeed, the conditional moment condition in~\eqref{eq:mot-dro-extended} gives rise to a continuum of constraints unless the $\sigma$-algebra~$\mc G$ is finitely generated. This makes OT-DRO with conditional moment constraints fundamentally harder than conventional OT-DRO and poses a major technical challenge for establishing strong duality. In addition, the conditional moment constraints in~\eqref{eq:mot-dro-extended} are reminiscent of martingale constraints. However, it is well known that primal-dual pairs of martingale OT problems may suffer from a non-vanishing duality gap even if the underlying distributions are univariate~\citep[Examples~8.1~\&~8.2]{beiglbock2017complete}. These observations provide little reason to expect that the assumptions of \Cref{thm:strong-duality} can be significantly relaxed. However, the compactness assumptions on~$\mc V$ and~$\mc W$ can be relaxed if $\hat\nu$ is discrete.

\begin{corollary}
    \label{cor:strong-duality}
    Suppose that~$c$ and~$f$ are real-valued functions satisfying Assumptions~\ref{ass:strong_duality}\,(a) and~\ref{ass:strong_duality}\,(b), $\mc G=\{\emptyset,(\mc V\times\mc W)^2\}$ is the trivial $\sigma$-algebra, $r>0$ and~$\hat \nu$ is a discrete distribution supported on the interior of~$\mc V\times\mc W$. Then, the supremum of the primal problem~\eqref{eq:mot-dro-extended} equals the infimum of a restricted variant of the dual problem~\eqref{eq:dual-mot}, where $\mc C_b(\mc V)$ is replaced with the family of all constant functions.
\end{corollary}

\begin{proof}
    The claim follows directly from standard results in semi-infinite duality theory~\citep{shapiro2001duality}; see, for example, \citep[Theorem~2.3]{li2022tikhonov} for details.
\end{proof}

\section{{Unifying \texorpdfstring{$\phi$}~-Divergence- and OT Discrepancy-Based DRO}}
\label{sec:cost_function}

In Section~\ref{sec:mot-express} we have seen that OT-DRO with conditional moment constraints is highly expressive and encapsulates popular DRO models as special cases. We now show that it also gives rise to new DRO models that simultaneously hedge against misspecification of the likelihood ratios and the actual outcomes of the unknown true probability distribution. To this end, we introduce a joint $\phi$-divergence and OT discrepancy.

\begin{definition}[Joint $\phi$-divergence and OT discrepancy]
If $\mc Z \subseteq \R^{d_z}$ is a convex closed set, $d:\mathcal{ Z \times Z}\rightarrow\R_+$ is a lower semicontinuous cost function satisfying the identity of indiscernibles, $\phi$ is an entropy function and $\theta\in \R_{++}^2$ a positive weight vector with $\theta_1^{-1} + \theta_2^{-1} = 1$, then the joint $\phi$-divergence and OT discrepancy induced by~$d$, $\phi$ and $\theta$ is the function $\mathds W^\theta_{d,\phi}: \mc P(\mc Z) \times \mc P(\mc Z)\rightarrow [0,+\infty]$ defined through
\begin{equation}
    \label{eq:phi-OT-divergence}
    \mathds W^\theta_{d,\phi}(\mu, \hat \mu) = \left\{ \begin{array}{cl}
    \inf & \theta_1\cdot \EE_{\gamma'}[ d(Z, \hat Z)] +\theta_2\cdot \mathds D_\phi (\gamma',\gamma)\\
    \rm{s.t.} & \gamma,\gamma' \in \mc P(\mc Z \times {\mc Z}), ~\gamma'_Z = \mu,~\gamma_{\hat Z} = \hat \mu.
    \end{array} \right.
\end{equation}
\label{def:joint-phi-OT-discrepancy}
\end{definition}

\begin{rmk}
\label{rem:joint-div-ot}
Note that $\mathds W^\theta_{d,\phi}(\mu, \mu)=0$. Indeed, if $\mu=\hat\mu$, then it is optimal to set $\pi=\pi'$ to the canonical self-coupling of~$\mu$ in~\eqref{eq:phi-OT-divergence}. In addition, if~$\mc Z$ is compact, then the infimum of problem~\eqref{eq:phi-OT-divergence} is attained for any $\mu,\hat\mu\in\mc P(\mc Z)$. Indeed, $\EE_{\pi'}[ d(Z, \hat Z)]$ and $\mathds D_\phi (\pi',\pi)$ are weakly lower semicontinuous in~$(\pi,\pi')$ thanks to \cite[Propositions~3.3 and~3.12]{Kuhn_Shafiee_Wiesemann_2025}. The compactness of~$\mc Z$ and the continuity of the coordinate projections further imply via \cite[Propositions~3.3 and~3.6]{Kuhn_Shafiee_Wiesemann_2025} that the feasible set of problem~\eqref{eq:phi-OT-divergence} is weakly compact. Hence, the infimum of problem~\eqref{eq:phi-OT-divergence} is attained.
\end{rmk}

Note that classical OT seeks one single coupling~$\gamma\in\Pi(\mu,\hat\mu)$ of~$\mu$ and~$\hat \mu$ that minimizes the transportation cost $\EE_\gamma[d(Z, \hat Z)]$. In contrast, the optimization problem in~\eqref{eq:phi-OT-divergence} seeks a {\em pair} of joint distributions $\gamma,\gamma'\in\mc P(\mc Z\times\mc Z)$, which form a `split' coupling in the sense that the first marginal of~$\gamma'$ equals~$\mu$, and the second marginal of~$\gamma$ equals~$\hat\mu$. The second term $\theta_2\cdot\mathds D_\phi (\gamma',\gamma)$ in the objective function penalizes disagreement between~$\gamma$ and~$\gamma'$, whereas the first term $\theta_1\cdot \EE_{\gamma'}[ d(Z, \hat Z)]$ is naturally interpreted as a standard transportation cost. The pair of weights $\theta=(\theta_1,\theta_2)$ regulates the relative cost of transporting mass versus the cost of allowing disagreement between~$\gamma'$ and~$\gamma$. As we will show formally in \Cref{lem:joint-ot-discperancy-to-classical}, when $\theta$ approaches~$(1,\infty)$ the mismatch penalty $\theta_2\cdot\mathds D_{\phi}(\gamma',\gamma)$ enforces the constraint~$\gamma'=\gamma$ such that~$\mathds W^\theta_{d,\phi}(\mu, \hat \mu)$ reduces to~$\mathds W_d(\mu,\hat\mu)$. In the opposite limit where $\theta$ approaches $(\infty,1)$, the transportation cost $\theta_1 \cdot \EE_{\gamma'}[d(Z,\hat Z)]$ forces $\gamma'$ to concentrate on the diagonal of the Cartesian product $\mc Z\times\mc Z$, where $z=\hat z$, such that~$\mathds W^\theta_{d,\phi}(\mu, \hat \mu)$ reduces to~$\mathds D_{\phi}(\mu,\hat\mu)$. Therefore, the joint $\phi$-divergence and OT discrepancy interpolates seamlessly between a pure OT discrepancy and a pure $\phi$-divergence. Increasing~$\theta_2$ encourages fidelity to the geometry of the outcome space encoded by the transportation cost function, while increasing~$\theta_1$ encourages fidelity to the likelihood structure encoded by the entropy~function.

\begin{lemma}
    The following identities hold for all $\mu, \hat \mu \in \mc P(\mc Z)$.
    \begin{enumerate}[label=\normalfont{(\roman*)}]
        \item $\lim_{\theta\to (1,\infty)}\mathds W_{d, \phi}^\theta (\mu, \hat \mu) = \mathds W_d(\mu, \hat \mu)$
        \item $\lim_{\theta\to (\infty,1)}\mathds W_{d, \phi}^\theta (\mu, \hat \mu) = \mathds D_\phi(\mu, \hat \mu)$
    \end{enumerate}
    \label{lem:joint-ot-discperancy-to-classical}
\end{lemma}
\begin{proof}
As~$\theta_2$ grows, the second term in the objective function of problem~\eqref{eq:phi-OT-divergence} forces~$\gamma'$ to approximate~$\gamma$. In the limit where~$\theta_1$ tends to~$1$ and~$\theta_2=\theta_1/(\theta_1-1)$ tends to~$\infty$, the objective function of~\eqref{eq:phi-OT-divergence} simplifies to~$\EE_\gamma[d(Z,\hat  Z)]$ if~$\gamma=\gamma'$; $=\infty$ otherwise. Thus, $\mathds W^\theta_{d,\phi}(\mu, \hat \mu)$ converges to~$\mathds W_d(\mu, \hat \mu)$. This proves assertion~(i).

As the cost function~$d$ is nonnegative and satisfies the identity of indiscernibles, the term $\theta_1 \mathbb E_{\gamma'}[d(Z,\hat Z)]$ in the objective function of~\eqref{eq:phi-OT-divergence} penalizes probability mass away from the diagonal $\{(z,\hat z): z=\hat z\}$. Thus, in the limit where~$\theta_1$ tends to~$\infty$ and~$\theta_2=\theta_1/(\theta_1-1)$ tends to~$1$, the only feasible limiting choice for~$\gamma'$ is the diagonal embedding of~$\mu$, which is defined as the pushforward distribution of~$\mu$ under the mapping $z\mapsto(z,z)$. Consequently, the first term in the objective function of~\eqref{eq:phi-OT-divergence} vanishes, and the overall objective function simplifies to $\mathds D_\phi(\gamma',\gamma)$ if~$\gamma'$ is the diagonal embedding of~$\mu$ and to~$\infty$ otherwise. Problem~\eqref{eq:phi-OT-divergence} thus reduces to
\begin{equation*}
    \inf_{\gamma\in\mathcal P(\mathcal Z\times \mc Z)} \left\{\mathds D_\phi(\gamma',\gamma): \gamma_{\hat Z}=\hat\mu\right\},
\end{equation*}
where $\gamma'$ denotes the diagonal embedding of $\mu$. As $\mu$ and $\hat \mu$ are obtained by pushing forward $\gamma'$ and $\gamma$ under the mapping $(z,\hat z)\mapsto \hat z$, respectively, the data processing inequality implies that the objective function~$\mathds D_\phi(\gamma',\gamma)$ of the resulting problem is bounded below by~$\mathds D_\phi(\mu, \hat \mu)$. This lower bound is attained if~$\gamma$ is set to the pushforward of~$\hat\mu$ under the mapping $z\mapsto(z,z)$. As this particular choice of~$\gamma$ is feasible in the above problem, its infimum must equal~$\mathds D_\phi(\mu, \hat \mu)$. This proves assertion~(ii). 
\end{proof}

If $\mathds D=\mathds W_{d,\phi}$ is a joint $\phi$-divergence and OT discrepancy, then the worst-case expectation problem in~\eqref{eq:dro} can be reformulated as an instance of problem~\eqref{eq:mot-dro}.

\begin{theorem}[Joint $\phi$-divergence and OT discrepancy-based DRO]
Consider an instance $(\ell, \mc Z, \hat \mu, \mathds D, r)$ of the worst-case expectation problem in~\eqref{eq:dro}, where $\mathds D=\mathds W^\theta_{d,\phi}$ is the joint $\phi$-divergence and OT discrepancy induced by~$d$, $\phi$ and~$\theta$. Set $\mathcal{V} = \mathcal{Z}$, $\mathcal{W}=\mathbb{R}_+$, $\mc G=\{\emptyset, (\mc V\times\mc W)^2\}$, $\hat \nu = \hat \mu \otimes \delta_1$, $f(v, w) = w \cdot \ell(v)$,  and 
\begin{equation}
    c((v,w), (\hat v, \hat w))= \theta_1 \cdot w \cdot  d(v, \hat v) +\theta_2 \cdot  \phi(w).
    \label{eq:interpolating-cost}
\end{equation} 
If $r>0$ and $\phi'_\infty=\infty$, then the inner worst-case expectation problem in~\eqref{eq:dro} has the same optimal value as the instance $(f, \mc V, \mc W, \mc G, \hat \nu, c, r)$ of problem~\eqref{eq:mot-dro}. 
\label{theorem:joint-phi-OT-discrepancy-DRO}
\end{theorem}

\begin{proof} 
As~$\phi'_\infty=\infty$, the discussion after \Cref{lemma:decomposition} implies that the $\phi$-divergence simplifies to $\mathds D_\phi(\gamma',\gamma) =\EE_\gamma[\phi(\frac{\diff\gamma'}{\diff\gamma}(Z, \hat Z))]$ for $\gamma'\ll\gamma$; $=\infty$ otherwise. Thus, we have
\begin{equation}
    \mathds W_{d, \phi}^\theta(\mu,\hat\mu)=\left\{\begin{array}{cl}
    \inf & \EE_{\gamma}\left[\theta_1\cdot \frac{\diff\gamma'}{\diff\gamma}(Z,\hat Z) \cdot d(Z, \hat Z) +\theta_2\cdot \phi\left( \frac{\diff\gamma'}{\diff\gamma}(Z, \hat Z) \right)\right] \\
    \rm{s.t.} & \gamma,\gamma' \in \mc P(\mc Z \times {\mc Z}), ~ \gamma'\ll\gamma,~\gamma'_Z = \mu,~\gamma_{\hat Z} = \hat \mu.
    \end{array}\right.
    \label{eq:joint-phi-ot-refor}
\end{equation}
For ease of notation, we henceforth use 
\begin{equation}
    \mathpzc W = \sup\limits_{\mu \in \mathcal P(\mc Z)} \left\{ \EE_\mu[\ell(Z)] : \mathds W^\theta_{d, \phi}(\mu, \hat \mu) \leq r\right\} 
    \label{eq:joint-phi-ot-worst-case}
\end{equation}
as shorthand for the worst-case expected loss of a fixed $\ell\in\mc L$ in~\eqref{eq:dro} and
\begin{equation}
    \mathpzc M = \sup\limits_{\nu \in \mathcal P(\mc V\times \mc W)} \left\{ \EE_{\nu}[f(V, W)] : \mathds M (\nu, \hat \nu) \leq r\right \}
    \label{eq:mot-worst-case}
\end{equation}
as shorthand for the worst-case expected loss in~\eqref{eq:mot-dro}. Note also that the cost function~$c$ defined in~\eqref{eq:interpolating-cost} satisfies \Cref{ass:growth} because~$\phi'_\infty=\infty$ and $\text{dom}(\phi)\subseteq \R_+$ imply that~$\phi$ grows superlinearly and because~$\hat W=1$ $\hat\nu$-almost surely. In the remainder we first show that $\mathpzc W \geq \mathpzc M$ (Step~1), and then we show that $\mathpzc M \geq \mathpzc W$ (Step~2).

\textit{Step 1 ($\mathpzc W \geq \mathpzc M$).} Choose any $\nu \in \mc P(\mc V \times \mc W)$ with $\mathds M(\nu, \hat \nu )\leq  r$. By \Cref{lemma:mot-attainment} in the appendix, the infimum of problem~\eqref{eq:mot} is thus attained by some $\pi \in \mc P((\mc V\times\mc W)^2)$ with $\pi_{(V, W)} = \nu$, $\pi_{(\hat V, \hat W)} = \hat \nu$, $\EE_{\pi}[W] = 1$ and $\EE_\pi[c((V, W), (\hat V, \hat W))] \leq r$. Next, we construct two probability distributions $\gamma, \gamma' \in \mc P(\mc Z \times \mc Z)$ with~$\gamma'\ll \gamma$ by setting $\gamma  = \pi_{(V, \hat V)}$ and specifying the Radon-Nikodym derivative of~$\gamma'$ with respect to~$\gamma$ as
\[
    \textstyle \frac{\diff \gamma'}{\diff \gamma}(V,\hat V) = \EE_{\pi}[W | V, \hat V]\quad \pi\text{-a.s.}
\]
Note that $\gamma'$ is indeed a probability distribution because 
\[
    \textstyle \EE_\gamma\left[ \frac{\diff \gamma'}{\diff \gamma}(V,\hat V) \right] = \EE_\pi\left[ \frac{\diff \gamma'}{\diff \gamma}(V,\hat V) \right] = \EE_\pi\left[\EE_{\pi}[W | V, \hat V] \right] = \EE_\pi\left[W\right]=1.
\]
In the following, we identify~$V$ with~$Z$ and~$\hat V$ with~$\hat Z$, and we interpret $\gamma$ and $\gamma'$ as distributions of~$(Z,\hat Z)$. We thus have $\gamma_{\hat Z} = \pi_{\hat V} = \hat \nu_{\hat V} = \hat \mu$, where the three equalities follow from the definitions of~$\gamma$, $\pi$ and~$\hat \nu$. Setting $\mu =\gamma'_Z\in \mc P(\mc Z)$, we also find
\[
    \textstyle \EE_\mu[\ell(Z)]  = \EE_\gamma \left[ \frac{\diff\gamma'}{\diff\gamma}(Z,\hat Z) \cdot \ell(Z) \right] = \EE_\pi \left[ W \cdot \ell(V) \right] = \EE_{\nu}[f(V, W)],
\]
where the three equalities follow from the construction of~$\mu$, $\gamma'$ and~$f$. Using a similar reasoning, one can further demonstrate that
\begin{align*}
    &\textstyle \EE_{\gamma}\left[\theta_1\cdot \frac{\diff\gamma'}{\diff\gamma}(Z,\hat Z) \cdot d(Z, \hat Z) +\theta_2\cdot \phi\left( \frac{\diff\gamma'}{\diff\gamma}(Z, \hat Z) \right)\right]\\
    =\,& \EE_{\gamma}\left[\theta_1\cdot \EE_{\pi}[W | V, \hat V] \cdot d(V, \hat V) +\theta_2\cdot \phi\left( \EE_{\pi}[W | V, \hat V] \right)\right] \leq  \EE_\pi \left[c((V, W), (\hat V, \hat W)) \right] \leq r,
\end{align*}
where the two inequalities follow from Jensen's inequality and from the feasibility of~$\pi$ in~\eqref{eq:mot-dro-extended}, respectively. These insights imply that $(\gamma, \gamma')$ is feasible in~\eqref{eq:joint-phi-ot-refor} with objective value at most~$r$, which implies that~$\mathds W_{d, \phi}^\theta(\mu,\hat\mu)\leq r$. This in turn ensures that~$\mu$ is feasible in~\eqref{eq:joint-phi-ot-worst-case} with objective value $\EE_{\mu}[\ell(Z)] = \EE_{\nu}[f(V, W)]$. As~$\nu$ is an arbitrary feasible solution of~\eqref{eq:mot-worst-case}, we may thus conclude that $\mathpzc W \geq \mathpzc M$.

\textit{Step 2 ($\mathpzc M \geq \mathpzc W$).} Choose any $\mu \in \mc P(\mc Z)$ with $\mathds W^\theta_{d, \phi}(\mu, \hat \mu) < r$. 
Thus, there exist distributions $\gamma, \gamma' \in \mc P(\mc Z\times\mc Z)$ with $\gamma' \ll \gamma$, $\gamma'_Z = \mu $, $\gamma_{\hat Z}=\hat\mu$ and
\[
    \textstyle \EE_{\gamma}\left[\theta_1\cdot \frac{\diff\gamma'}{\diff\gamma}(Z,\hat Z) \cdot d(Z, \hat Z) +\theta_2\cdot \phi\left( \frac{\diff\gamma'}{\diff\gamma}(Z, \hat Z) \right)\right] \leq r.
\]
In the following, we identify~$Z$ with~$V$ and~$\hat Z$ with~$\hat V$, and we interpret~$\gamma$ and~$\gamma'$ as distributions of~$(V,\hat V)$. We can now define $\pi \in \mc P((\mc V\times\mc W)^2)$ as the pushforward distribution of~$\gamma$ under the transformation $(V,\hat V)\mapsto (V, \frac{\diff\gamma'}{\diff\gamma}(Z,\hat Z), \hat V, 1)$. By construction, we thus have $\pi_{(\hat V, \hat W)} = \gamma_{\hat Z}\otimes \delta_1 = \hat \mu\otimes \delta_1$. Setting $\nu = \pi_{(V, W)}$, we also find
\begin{align*}
    \textstyle \EE_{\nu}[f(V,W)] = \EE_{\pi}[W\cdot \ell(V)] = \EE_{\gamma} \left[ \frac{\diff\gamma'}{\diff\gamma}(Z,\hat Z) \cdot \ell(Z) \right] = \EE_\mu[\ell(Z)],
\end{align*}
where the first equality exploits the definitions of~$\nu$ and~$f$, the second equality follows from the definition of~$\pi$ and the change of variable formula, and the last equality holds because~$\gamma_Z=\mu$. Using similar arguments, one can show that $\EE_\pi[W]=1$ and that
\begin{align*}
    \mathds M(\nu, \hat \nu) 
    &\leq \EE_\pi \left[c((V, W), (\hat V, \hat W)) \right] \\
    &= \textstyle \EE_{\gamma}\left[\theta_1\cdot \frac{\diff\gamma'}{\diff\gamma}(Z,\hat Z) \cdot d(Z, \hat Z) +\theta_2\cdot \phi\left( \frac{\diff\gamma'}{\diff\gamma}(Z, \hat Z) \right)\right] \leq r,
\end{align*}
where the first inequality holds because~$\pi$ is feasible in~\eqref{eq:mot}, the equality follows from the definition of~$c$ and the change of variable formula, and the second inequality holds because $(\gamma, \gamma')$ is feasible in~\eqref{eq:joint-phi-ot-refor}. Hence, $\nu$ is feasible in \eqref{eq:mot-worst-case} with objective value $\EE_\nu[f(V, W)] = \EE_\mu[\ell(Z)]$. As~$\mu$ is an arbitrary strictly feasible solution of~\eqref{eq:joint-phi-ot-worst-case} and as the supremum of~\eqref{eq:joint-phi-ot-worst-case} is attained by a sequence of strictly feasible solutions,\footnote{Note that as~$r>0$, any~$\mu$ feasible in~\eqref{eq:joint-phi-ot-worst-case} can be approximated by a sequence of distributions $\mu_n=(1-\frac{1}{n})\mu+\frac{1}{n}\hat \mu$, all of which satisfy $\mathds W^\theta_{d, \phi}(\mu, \hat \mu) < r$ because $\mathds W^\theta_{d, \phi}$ is convex and $\mathds W^\theta_{d, \phi}(\hat\mu,\hat\mu)=0$.} we may thus conclude that $\mathpzc M \geq \mathpzc W$. This observation completes the proof.
\end{proof}

\Cref{theorem:joint-phi-OT-discrepancy-DRO} shows that the joint $\phi$-divergence and OT discrepancy DRO model can be interpreted as a generalized OT-DRO problem by reducing the underlying worst-case expectation problem to an instance of~\eqref{eq:mot-dro}. 
We will see below that this reformulation often admits a finite conic programming reformulation, allowing for efficient computation using standard off-the-shelf optimization solvers.

Next, we dualize the joint $\phi$-divergence and OT discrepancy-based DRO problem.

\begin{theorem}
\label{thm:interpolating-phi-wass-refor}
If all assumptions of \Cref{theorem:joint-phi-OT-discrepancy-DRO} hold, $\hat\mu$ is a discrete distribution supported on the interior of~$\mc Z$, and the conjugate~$\phi^*$ of~$\phi$ is non-decreasing, then the inner maximization problem in~\eqref{eq:dro} has a strong dual equivalent~to
\begin{equation}
    \inf_{\lambda \in \R_+,\,\psi \in\mathbb{R}}  \lambda r - \psi+  \lambda \theta_2 \cdot \EE_{\hat \mu}\left[\phi^{*}\left(\frac{\ell_{\lambda \theta_1}(\hat Z) + \psi}{\lambda \theta_2 }  \right)\right],
    \label{eq:interpolating-kl-wass}
\end{equation}
where 
the $d$-transform of $\ell$ with step size $\alpha \geq 0$ is the function $\ell_\alpha:\mc Z\rightarrow \overline{\R}$ defined via  
\[
    \ell_{\alpha}(\hat z) = \sup_{z \in \mc Z} \big\{\ell(z)- \alpha \cdot d(z, \hat z)\big\}.
\]
\end{theorem}

\begin{proof}
By \Cref{theorem:joint-phi-OT-discrepancy-DRO}, the inner problem in~\eqref{eq:dro} is equivalent to~\eqref{eq:mot-dro}. Note that
\Cref{ass:strong_duality}\,(a) is satisfied thanks to the definition of~$c$ in~\eqref{eq:interpolating-cost}, the lower semicontinuity of~$d$ and the continuity of~$\phi$. Similarly, \Cref{ass:strong_duality}\,(b) holds because of the definitions of~$f$ and~$\hat \nu$. Finally, $\hat\nu=\hat\mu\otimes\delta_1$ is a discrete distribution supported on the interior of~$\mc V\times\mc W$. \Cref{cor:strong-duality} thus implies that problem~\eqref{eq:mot-dro-extended} has a strong dual akin to~\eqref{eq:dual-mot}, where~$\mc C_b(\mc G)$ is replaced with~$\mathbb R$. By the definitions of~$f$, $c$ and~$\hat \nu$, the nonnegativity of the entropy function~$\phi$ and the normalization condition~$\phi(1)=0$, this dual problem is given by
\begin{align*}
    & \inf_{\lambda \in \R_+,\, \psi \in\R} \lambda r - \psi + \EE_{\hat \nu }\left[  \sup_{v \in \mc V, w \in \mc W} \left( \ell(v) +\psi-\lambda \theta_1 \cdot  d(v,\hat V)\right)\cdot w   -\lambda\theta_2\cdot \phi(w) \right]\\
    = &  \inf_{\lambda \in \R_+,\,\psi \in\R} \lambda r - \psi  +  \EE_{\hat \mu} \left[\lambda\theta_2 \cdot \phi^{*}\left(\frac{1}{\lambda\theta_2}\sup_{z \in \mc Z} \ell(z )+ \psi-\lambda\theta_1\cdot d(z, \hat Z) \right) \right].
\end{align*}
Here, the equality holds because $\text{dom}(\phi)\subseteq\R_+=\mc W$, which implies that the maximum over~$w$ evaluates the conjugate of $\lambda\theta_2\cdot \phi(w)$, because the conjugate of a nonnegative multiple of a proper convex function coincides with the perspective of the conjugate of this function \citep[Theorem~16.1]{Rockafellar1970} and because~$\phi^*$ is non-decreasing by assumption. The claim finally follows from the definition of the $d$-transform of the loss function.
\end{proof} 

Problem~\eqref{eq:interpolating-kl-wass} simplifies significantly when $\phi$ is set to the entropy function of the Kullback-Leibler divergence (cf.~Table~\ref{tab:phi-divergence}), which has a non-decreasing conjugate. 

\begin{corollary}
\label{cor:interpolating-kl-wass}
If $\phi(t)=t\log t-t+1$, then problem~\eqref{eq:interpolating-kl-wass} is equivalent to
\begin{equation}
    \inf_{\lambda\in \R_+}\lambda r +\lambda \theta_2 \cdot \log \left(\EE_{\hat \mu}\left[ \exp\left(\frac{\ell_{\lambda\theta_1}(\hat Z)}{\lambda \theta_2}\right)\right]\right). 
    \label{eq:interpolating-kl-wass2}
\end{equation}
\end{corollary}
\begin{proof}
A direct calculation shows that $\phi^*(t)=\exp(t)-1$, which is convex and non-decreasing. For the special entropy function at hand, the unconstrained convex minimization over~$\psi$ in~\eqref{eq:interpolating-kl-wass} can thus be solved analytically for any fixed~$\lambda\in\mathbb R_+$. Indeed, the unique solution of the corresponding first-order optimality condition is
\[
    \psi^\star = - \lambda \theta_2 \cdot \log \left(\EE_{\hat \mu}\left[ \exp\left(\frac{\ell_{\lambda \theta_1}(\hat Z)}{\lambda \theta_2}\right)\right]\right) .
\]
Substituting~$\psi\opt$ into~\eqref{eq:interpolating-kl-wass} yields~\eqref{eq:interpolating-kl-wass2}.
\end{proof}
\begin{rmk}[Choice of~$\theta$]
By tuning~$\theta$ in problem~\eqref{eq:interpolating-kl-wass2}, one can interpolate between the dual forms of Kullback-Leibler divergence-based DRO (KL-DRO) and Wasserstein DRO. This observation is expected in view of \Cref{lem:joint-ot-discperancy-to-classical}. 
\begin{enumerate}[label=\normalfont(\alph*)]
    \item If $\theta_1 \nearrow \infty$ and $\theta_2\searrow 1 $, then the infimum of problem \eqref{eq:interpolating-kl-wass2} converges to
    \[
        \min\limits_{\lambda \in \R_+} \lambda r + \lambda \log\left(\EE_{\hat \mu} \left[\exp \left(\frac{\ell(\hat Z)}{\lambda }\right)\right]\right),
    \]
    which is precisely the dual KL-DRO problem \citep[Theorem 1]{hu2013kullback}.
    \item If $\theta_1 \searrow 1$ and $\theta_2 \nearrow \infty$, then the infimum of problem \eqref{eq:interpolating-kl-wass2} converges to
    \[
        \min\limits_{\lambda \in \R_+} \lambda r + \EE_{\hat \mu} \left[ \sup_{z \in \mc Z} \big\{\ell(z)- \lambda d(z, \hat Z)\big\} \right] ,
    \]
    which is precisely the dual OT-DRO problem \cite[Theorem~1]{ref:blanchet2019quantifying}. 
\end{enumerate}
\end{rmk}

We now show that the worst-case expectation problem \eqref{eq:interpolating-kl-wass2} can be transformed into a finite convex program for many practical loss functions. Specifically, we consider a worst-case expectation problem with $\ell(z) = \max_{k\in[K]}\ell_k(z)$, where $\ell_k:\R^{d_z} \rightarrow \overline \R$ are concave functions for all~$k \in [K]$. The representation of~$\ell$ as the pointwise maximum of concave functions does not sacrifice much modeling power; see, e.g.,~\citep{mohajerin2018data}.

\begin{theorem}[Convex reduction]
\label{thm:tractable}
If the support set $\mc Z \subseteq \R^{d_z}$ is convex and closed, the negative constituent functions $-\ell_k$ are proper, convex and closed for all~$k\in[K]$, the transportation cost is given by $d(z,\hat z) = \|z-\hat z\|_p$ for some $p\in[1,+\infty]$, and $\hat\mu=\frac{1}{n}\sum_{i=1}^n\delta_{\hat z_i}$, then problem \eqref{eq:interpolating-kl-wass2} is equivalent to the finite convex program
\begin{equation*}
 \begin{array}{cclll}
         &\min  &\lambda r  +t & \\
         &\suchthat  & \lambda \in \R_+, ~t \in \R, ~\eta \in \R_+^n, ~\zeta\in \R^n, ~\xi\in \R^{nK},~ \omega\in \R^{nK}\\
         && (\eta_i, \lambda \theta_2, \zeta_i - t) \in \mc K_{\exp} &\forall i \in [n] \\
         &&(-\ell_k)^*(\xi_{ik}-\omega_{ik}) +\sigma_{\mc Z}(\omega_{ik})-\xi_{ik}^T\hat{z}_i \leq \zeta_i  &\forall k \in [K], \,  \forall i \in [n]\\
           &&  \|\xi_{ik}\|_q \leq \lambda\theta_1 &\forall k \in [K], \, \forall i \in [n]\\
             &&\frac{1}{n}\sum_{i=1}^n \eta_i \leq \lambda \theta_2,
    \end{array}
    \label{eq:interpolating-mot-piecewise-concave}
\end{equation*}
where $\tfrac{1}{p}+\tfrac{1}{q}=1$, $\sigma_{\mc Z}$ is the support function of~$\mc Z$, and
the exponential cone is given~by 
\begin{small}
\begin{equation*}
    \mc K_{\exp} = {\rm{cl}}\left(\Big\{(x_1,x_2,x_3)\in \R^3: x_1\ge x_2 \cdot \exp\left(\tfrac{x_3}{x_2}\right),x_2>0\Big\} \right). 
\end{equation*}
\end{small}
\end{theorem}
\begin{proof}
By introducing an epigraphical variable~$t\in \R$, we can reformulate~\eqref{eq:interpolating-kl-wass2} as
\begin{align}
    & \left\{ \begin{array}{clll}
    \min  &\lambda r  +t & \\
    \suchthat  & \lambda \in \R_+, ~t \in \R\\
    & \lambda \theta_2 \log \left(\EE_{\hat \mu}\left[ \exp\left(\frac{\ell_{\lambda\theta_1}(\hat Z)}{\lambda \theta_2}\right)\right]\right) \leq t
    \end{array}\right. \nonumber 
    \\ = \, &\left\{ \begin{array}{clll}
    \min  &\lambda r  +t & \\
    \suchthat & \lambda \in \R_+, ~t \in \R,~ \eta \in \R_+^n,~ \zeta\in \R^n\\
    & (\eta_i, \lambda \theta_2, \zeta_i - t) \in \mc K_{\exp} &\forall i \in [n] \\
    & \sup\limits_{z\in \mc Z} \Big\{  \max\limits_{k\in[K]} \ell_k(z) -\lambda \theta_1 \cdot d(z,\hat z_i)\Big\}\leq \zeta_i  &\forall i \in [n]\\[2ex]
    &\frac{1}{n}\sum_{i=1}^n \eta_i \leq \lambda \theta_2.
    \end{array}\right.
    \label{eq:reformu2}
\end{align}
Note that the inequality constraint in the first problem shown above is equivalent to $\EE_{\hat\mu }[ \exp((\ell_{\lambda\theta_1}(\hat Z)-t) / (\lambda \theta_2))]\leq 1$. This constraint can be decomposed into~$n$ exponential cone constraints and one linear constraint by introducing an auxiliary decision variable~$\eta \in \R_+^n$. Problem~\eqref{eq:reformu2} is thus obtained by introducing an auxiliary variable $\zeta_i\geq \ell_{\lambda \theta_1}(\hat z_i)$ for every $i\in[n]$ and by replacing the $d$-transform~$\ell_{\lambda \theta_1}(\hat z_i)$ by its definition. The claim then follows by applying standard reformulations of the resulting robust constraints as in~\cite[Theorem~4.2]{mohajerin2018data}.
\end{proof}

If the loss function~$\ell$ is piecewise linear and the transportation cost function~$d$ is the squared Euclidean norm, then problem~\eqref{eq:interpolating-kl-wass2} reduces to a tractable conic program.

\begin{corollary}[Piecewise linear loss and quadratic cost]
\label{coro:piece-lir}
$\!\!\!$If $\mc Z =\R^{d_z}$, $d(z,\hat z) = \|z-\hat z\|^2$ is quadratic and $\ell(z) = \max_{k \in [K]} a_k^\top z+b_k$ is piecewise linear, then problem~\eqref{eq:interpolating-kl-wass2} is equivalent to the tractable conic program
\begin{equation*}
 \begin{array}{cclll}
         &\min  &\lambda r  +t & \\
         &{\rm{s.t.}}   & \lambda \in \R_+, t \in \R, ~\eta \in \R_+^n, ~\zeta\in \R^n\\
         && (\eta_i, \lambda \theta_2, \zeta_i - t) \in \mc K_{\exp} &\forall i \in [n] \\
         &&a_k^T\hat z_i+b_k + \frac{\|a_k\|^2}{4\lambda \theta_1}\leq \zeta_i  &\forall k \in [K], \,  \forall i \in [n]\\
        &&\frac{1}{n}\sum_{i=1}^n \eta_i \leq \lambda \theta_2.
    \end{array}
\end{equation*} 
\end{corollary}
\begin{proof}
Following the same steps as in the proof of \Cref{thm:tractable}, we show that~\eqref{eq:interpolating-kl-wass2} is equivalent to~\eqref{eq:reformu2} of. Next, the $i$-th robust constraint in~\eqref{eq:reformu2} can be recast as
\begin{align*}
    & \sup\limits_{z\in \R^d} \left\{\max\limits_{k\in[K]} a_k^\top z+b_k -\lambda \theta_1 \|z -\hat{z}_i\|^2\right\} \leq \zeta_i \\
    \iff  ~& \sup\limits_{z \in  \R^d} \Big\{ a_k^\top z+b_k- \lambda \theta_1 \|z -\hat{z}_i\|^2\Big\}\leq \zeta_i\quad\forall k\in[K]\\
    \iff ~&  a_k^T\hat z_i+b_k + \frac{\|a_k\|^2}{4\lambda \theta_1}\leq \zeta_i\quad\forall k\in[K],
\end{align*}
where the first equivalence holds because the supremum over~$v$ and the maximum over~$k$ can be interchanged, and the second equivalence is obtained by solving the convex quadratic program over~$v$ in closed form. The claim then follows by substituting the resulting constraint back into~\eqref{eq:reformu2} for every~$i\in[n]$.
\end{proof}

\section{Numerical Results}
\label{sec:visu}
We now show numerically that the joint $\phi$-divergence and OT discrepancy DRO model of Section~\ref{sec:cost_function} can outperform KL-DRO and Wasserstein DRO, even in the presence of noise or adversarial perturbations. Throughout this section we investigate a synthetic support vector machine (SVM) problem with training dataset $\{\hat z_i\}_{i=1}^n$. Each training sample $\hat z_i = (\hat x_i, \hat y_i)$ consists of a feature vector $\hat x_i \in \mathbb R^{d_x}$ and a label $\hat y_i \in \{+1,-1\}$. For a linear classifier parametrized by $(\beta,b) \in \mathbb R^{d_x} \times \mathbb R$, the prediction loss of a random sample $Z=(X,Y)$ is measured by the hinge loss $\ell_{(\beta,b)}(Z)=\max\{1 - Y(\beta^\top X + b),0\}$. We construct a distributionally robust SVM that minimizes the worst case expected hinge loss over all distributions of~$Z$ in an $r$-neighborhood of the empirical distribution $\hat\mu = \frac{1}{n}\sum_{i=1}^n \delta_{\hat z_i}$ with respect to the joint $\phi$ divergence and OT discrepancy of \Cref{def:joint-phi-OT-discrepancy}. In all experiments we use the transportation cost $d(z,\hat z) = \|x-\hat x\|_\infty^2 + \infty \cdot |y-\hat y|$, which forbids label changes, and choose $\phi$ to be the entropy function of the Kullback Leibler divergence. As the hinge loss is piecewise linear and we use the $\infty$-norm to quantify transportation costs in the feature space, the DRO problem may admit multiple optimal solutions. To promote numerical stability and consistent selection of a minimizer in degenerate cases, we add a small regularization term $10^{-3}\times \|\beta\|_2$ to the hinge loss. The source code of our experiments is available from \url{https://github.com/BaharTaskesen/Unifying-DRO}.

We use the following procedure to generate synthetic datasets. First, we sample an $s^\star$-sparse ground-truth parameter vector~$\beta^\star \in \R^{d_x}$ for some $s^\star \in[d_x]$. Specifically, we construct~$\beta^\star$ by first drawing a random vector from~$\mathcal N(0,I_{d_x})$ and then setting all but~$s^\star$ randomly selected entries of this vector to~$0$. Next, we generate $10{,}000$ independent feature vectors $\hat x_i \sim \mathcal N(0,I_{d_x})$. Labels are generated according to a noisy linear rule, that is, we set $\hat y_i = \operatorname{sign}(\beta^{\star\top}\hat x_i + \xi_i)$, where the disturbances~$\xi_i \sim \mathcal N(0,0.1)$ are independent of each other as well as independent of the features. The first $n$ samples are used for training, and the rest for testing. All reported results are averaged over $10$ repetitions of the same experiment with independent datasets.

\textbf{Influence of the Interpolation Parameter.}
\Cref{fig:theta-hyperpar} reports the correct classification rate (CCR) of the joint $\phi$-divergence and OT DRO model with $d_x=100$, $s^\star=10$, $r=1$ and $n=200$ as a function of $\theta_1 \in [1,10^3]$. The parameter $\theta_2$ is chosen such that $\theta_1^{-1} + \theta_2^{-1} = 1$. Lines and shaded areas represent averages and ranges over the $10$ trials, respectively. We observe that intermediate values of $\theta_1$ yield the best performance, outperforming both the Kullback-Leibler divergence regime corresponding to~$\theta_1 \gg 1$, where robustness is primarily driven by likelihood reweighting, and the Wasserstein distance regime corresponding to $\theta_1 \gtrsim 1$, where robustness is driven by geometric perturbations. These results indicate that, in the presence of label noise and limited training data, interpolating between divergence-based and Wasserstein distance-based robustness can substantially improve generalization performance.


\textbf{Influence of the Radius.}
We now assess how the out-of-sample performance of the different distributionally robust classifiers depends on the radius~$r$ of the underlying ambiguity set. Here, we set $d_x=100$, $s^\star=10$ and $n=200$. The interpolation parameter~$\theta$ of the joint $\phi$-divergence and OT discrepancy-based SVM is is tuned by $5$-fold cross validation. In addition, we construct Wasserstein distance-based and Kullback Leibler divergence-based SVMs by setting $\theta=(1,\infty)$ and $\theta=(\infty,1)$, respectively. 
Each of the three DRO models is solved for~$10$ radii~$r$ on a logarithmically spaced grid covering the interval $[10^{-5}, 1]$. \Cref{fig:ccr-radius} reports the CCR of each model on the test set as a function of~$r$. It also shows the CCR of a non-robust empirical SVM. Lines and shaded areas represent averages and ranges over the $10$ trials. As~$r$ tends to~$0$, all DRO models collapse to the empirical SVM and thus display the same CCR. At larger values of~$r$, however, the geometry of the ambiguity set matters. The Wasserstein and joint $\phi$-divergence and OT DRO models exhibit a marked improvement in CCR, reflecting their ability to exploit structural knowledge about the sparsity of~$\beta^\star$, which is captured by the $\infty$-norm transportation cost. In contrast, the Kullback Leibler divergence-based model, which only re-weights samples without modifying their locations, shows limited sensitivity to the radius and remains close to the empirical baseline. For larger radii, the performance of the Wasserstein model deteriorates sharply, indicating excessive conservatism. By contrast, the joint $\phi$-divergence and OT model degrades more gradually and maintains the highest CCR across a wide range of radii. This behavior suggests that combining OT-based perturbations with likelihood re-weighting leads to a more balanced form of robustness.

\begin{figure}
\centering

\begin{subfigure}{0.31\textwidth}
  \centering
  \includegraphics[width=\linewidth]{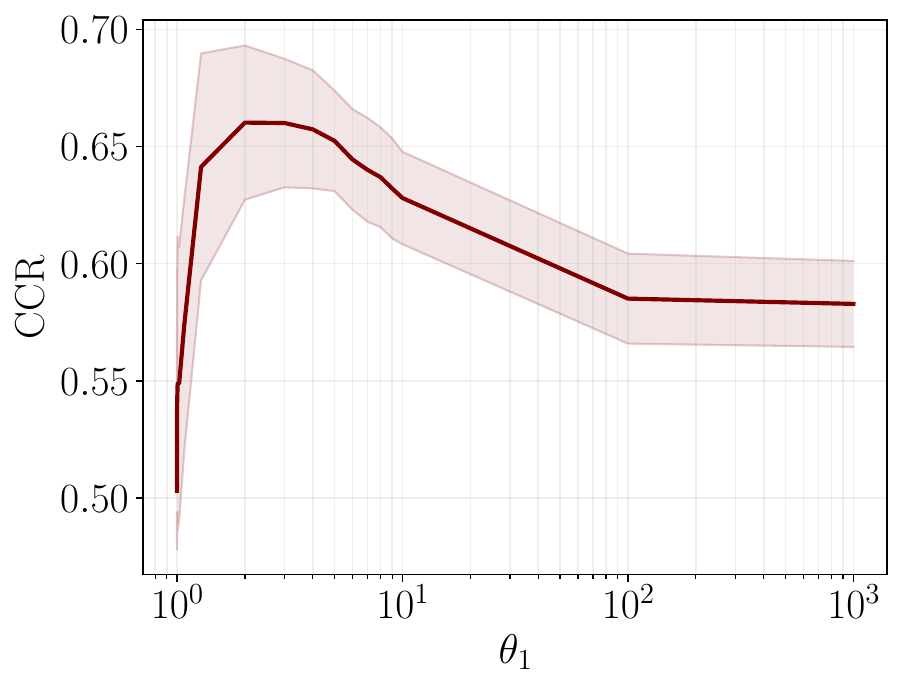}
  \caption{}
  \label{fig:theta-hyperpar}
\end{subfigure}\hfill
\begin{subfigure}{0.31\textwidth}
  \centering
  \includegraphics[width=\linewidth]{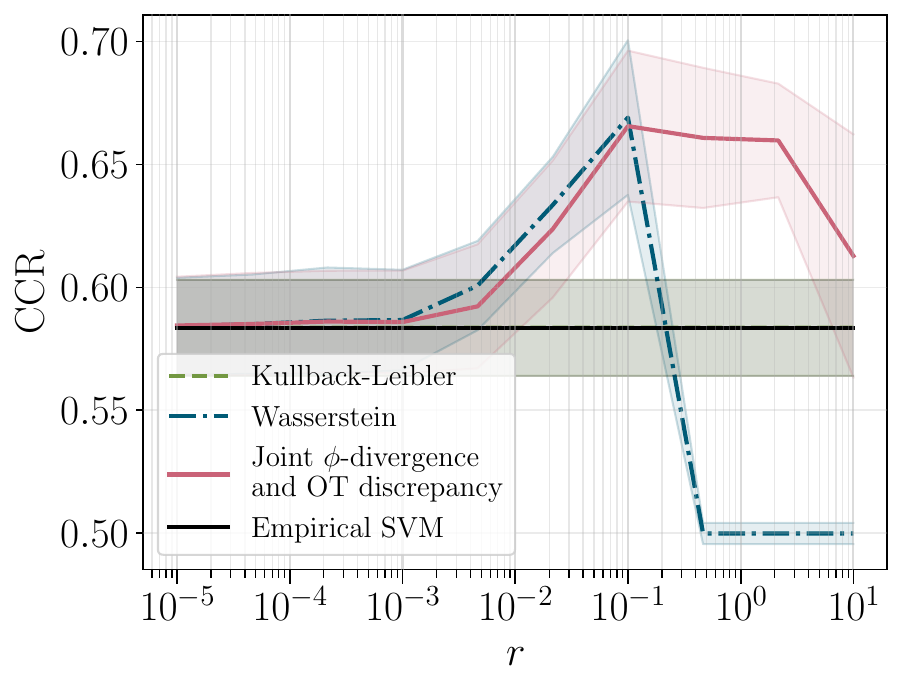}
  \caption{}
  \label{fig:ccr-radius}
\end{subfigure}\hfill
\begin{subfigure}{0.31\textwidth}
  \centering
  \includegraphics[width=\linewidth]{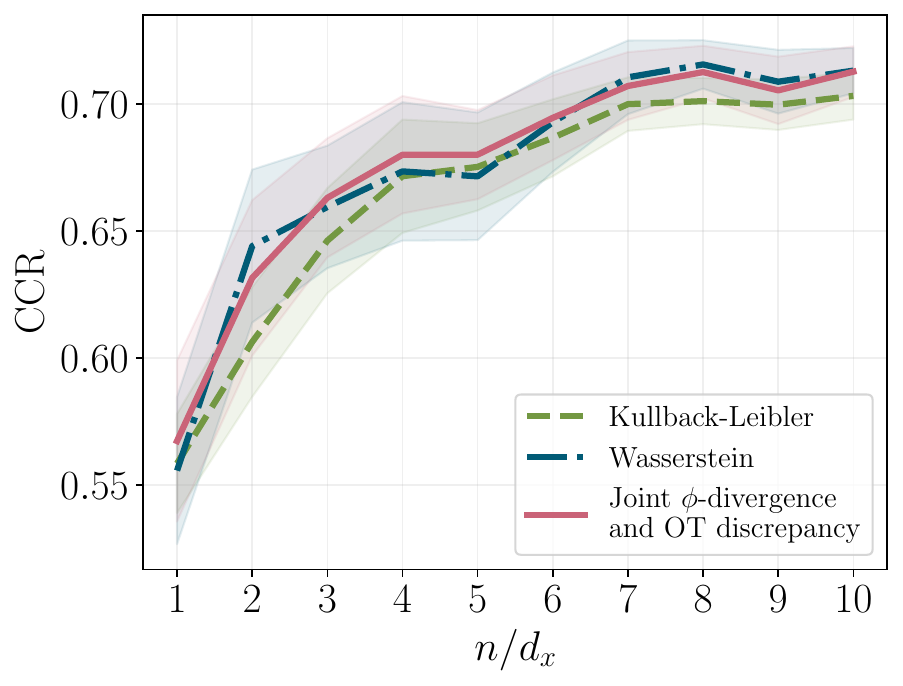}
  \caption{}
  \label{fig:n-d}
\end{subfigure}

\caption{Averages (solid lines) and ranges (shaded areas) of the CCR for the joint $\phi$-divergence and OT discrepancy-based SVM and several baselines, shown as functions of $\theta_1$, $r$ and $n$.}
\label{fig:ccr}
\end{figure}

\textbf{Influence of the Sample Size.}
Lastly, we investigate how the CCR of the different DRO models depends on~$n$. In this experiment we set~$d_x=10$ and $s^\star=5$. As before, the parameter $\theta$ of the joint $\phi$-divergence and OT DRO model is tuned via $5$-fold cross validation. For each value of~$n \in \{10,20,30,40,50, 60, 70, 80,90,100\}$ and for each DRO model, we select the best-performing radius~$r$ from a grid of~$10$ logarithmically spaced values in~$[10^{-5},1]$. The resulting averages and ranges of the test CCR over {$10$} independent trials are visualized in \Cref{fig:n-d} as a function of~$n/d_x$ (where~$d_x$ is fixed). We observe that the CCR of all models increases with~$n$. The joint $\phi$-divergence and OT discrepancy DRO model consistently achieves the highest or near-highest CCR, but the Wasserstein DRO model remains competitive across the entire low data regime that is shown ({$n\leq 100$}). The Kullback-Leibler divergence-based model, however, is strictly dominated even at larger sample sizes. These observations suggest that, when the feature dimension is small and there is exploitable geometric structure, OT-based robustness remains effective.

Overall, this experiment highlights that the relative advantage of different ambiguity sets depends on the interaction between sample size, feature dimension, and label noise, and that the proposed joint $\phi$-divergence and  formulation provides stable and robust performance across these regimes.

\appendix
\section{Technical Lemmas}
We first establish a basic result about conditional expectations, and we include a proof to keep this paper self-contained.
\begin{lemma}
\label{lemma:half-line}
    If $(X,Y)\sim\pi \in\mc P(\R^2)$ and $\EE_\pi[|Y|] < \infty$, then
    \begin{equation}
        \EE_\pi[Y | X] =0~\pi\text{-a.s.}\quad \iff \quad \EE_\pi[Y \boldsymbol 1_{\{X \leq  a\}}] = 0~\forall a\in \R.
    \label{eq:half-line-exp}
    \end{equation}
\end{lemma}
\begin{proof}
If $\EE_\pi[Y | X] =0$ $\pi$-almost surely, then the tower property of conditional expectations readily ensures that $\EE_\pi[Y \boldsymbol 1_{\{X \leq  a\}}] = \EE_\pi[ \EE_\pi[Y | X]\boldsymbol 1_{\{X \leq  a\}}] = 0$ for all fixed~$a\in \R$. This implies sufficiency. As for necessity, assume that $\EE_\pi[Y \boldsymbol 1_{\{X \leq  a\}}]=0$ for all $a\in \R$, and define a signed measure~$\mu$ on the Borel $\sigma$-algebra $\mathcal B(\R)$ through
\[
    \mu(B) = \EE_\pi\left[Y \boldsymbol 1_{\{X \in B\}}\right] \quad \forall B \in \mathcal B(\R).
\]
By assumption, we have $ \mu((-\infty, a]) =0$ for all $a\in \R$. Consequently, we also have
\[
    \mu(\R) = \EE_\pi[Y] = \lim\limits_{a\to \infty} \EE_\pi\left[ Y \boldsymbol 1_{\{X \leq a\}}\right] = 0
\]
by the dominated convergence theorem, which applies because $\EE_\pi[|Y|] < \infty$. Next, we define the family $\mathcal A = \{(-\infty, a]: a\in \R\}$ of half-lines, which is closed under finite intersections and thus constitutes a $\pi$-system. By \cite[Lemma~4.19]{aliprantis2006infinite}, $\mathcal A$ generates~$\mathcal B(\R)$. We also define the family $\mathcal L = \{B \in \mathcal B(\R) :  \mu(B) =0\}$ of $\mu$-null sets. As~$\mu(\R)=0$, it is clear that if~$B \in \mathcal L$, then $B^c\in\mc L$ because $ \mu(B^c) =  \mu(\R) - \mu(B) = 0$. Hence, $\mc L$ is closed under complements. Also, if $B_i\in \mathcal L$, $i \in \mathbb N$, are pairwise disjoint sets, then $\cup_{i\in \mathbb N} B_i\in\mc L$ because $ \mu(\cup_{i\in \mathbb N} B_i)  = \sum_{i \in \mathbb N} \mu(B_i) =0$. Hence, $\mc L$ is closed under countable unions of pairwise disjoint sets. In summary, $\mathcal L$ is a $\lambda$-system. As~$\mathcal A \subseteq \mathcal L$, the $\pi$-$\lambda$ theorem \cite[Lemma~4.11]{aliprantis2006infinite} implies that $\sigma(\mathcal A) = \mathcal B(\R) \subseteq \mathcal L$, that is, $\mu(B) =0 $ for every  $B \in \mathcal B(\R)$. For any bounded Borel function $\varrho: \R \to \R$ we thus have
\[
    0= \int_\R \varrho(x)\, \mu(\diff x) = \EE_\pi[ Y \varrho(X)] = \EE_\pi[\EE_\pi[ Y| X]\, \varrho(X)]],
\]
which implies that $\EE_\pi[Y | X] =0$ $\pi$-almost surely.
\end{proof}

We can now show that martingale properties are preserved under weak limits.

\begin{lemma}
\label{lm:martinagle_weak_converge} 
Let $\{\pi_n\}_{n\in\N}$ be a
sequence in $\mc P (\R^2)$ that converges weakly to $\pi$. If
\begin{enumerate}[label=\normalfont(\alph*)]
    \item\label{mart-n}   $\EE_{\pi_n}[Y| X]=X$ $\pi_n$-a.s.
          for every $n\in\N$ and
    \item\label{unif-mom} $\sup_{n\in\N}
           \EE_{\pi_n} [\| (X,Y)\|^{1+\delta}]< \infty$ for some $\delta >0$,
\end{enumerate}
then the weak limit retains the martingale property, that is, $\EE_{\pi}[Y| X]=X$ $\pi$-a.s.
\end{lemma}
\begin{proof}
    As $\pi_n$ converges weakly to $\pi$, Skorokhod’s representation theorem ensures that there exists a probability space on which random vectors $(X_n,Y_n)\sim\pi_n$ and $(X,Y)\sim\pi$ are defined, and where $(X_n,Y_n)$ converges to $(X,Y)$ almost surely. Let~$U$ be a standard Gaussian random variable independent of all other random variables on the same probability space. From now on we use $\EE[\cdot]$ to denote the expectation operator on this space. Assumption~\ref{mart-n} implies that $\EE[Y_n| X_n]=X_n$ almost surely, while Assumption~\ref{unif-mom} implies that the pair $(X_n,Y_n)$ has uniformly bounded moments of order $1+\delta$. A straightforward generalization of \cite[Proposition~1] {blanchet2024empirical} from $\delta=1$ to any $\delta>0$ implies that $(Y_n+U,X_n+U)$ also forms a martingale pair, that is, 
    \[
        \EE[Y_n+U| X_n+U]=X_n+U ~\text{a.s.}\quad \iff\quad \EE[(Y_n - X_n)\mathbf{1}_{\{X_n+U<a\}}]=0 ~ \forall a\in\R.
    \]
    Here, the equivalence follows from \Cref{lemma:half-line}, which applies due to assumption~\ref{unif-mom}. Next, note that the sequence $\{(Y_n-X_n) \mathbf{1}_{\{X_n+U<a\}}\}_{n\in\N}$ is uniformly integrable by~\ref{unif-mom}. Also, as $(X_n,Y_n,U)$ converges almost surely to $(X,Y,U)$ and as the probability of the event $\{X+U=a\}$ is zero, $(Y_n-X_n) \mathbf{1}_{\{X_n+U<a\}}$ converges to $(Y- X)\mathbf{1}_{\{X+U<a\}}$  almost surely. We may thus use Vitali's theorem to conclude that
    \[
        \EE[(Y - X)\mathbf{1}_{\{X+U<a\}} ]= \lim_{n\to\infty } \EE[(Y_n - X_n)\mathbf{1}_{\{X_n+U<a\}}]=0\quad \forall a\in\R.
    \]
    Hence, we have $\EE[Y+U| X+U]=X+U$ almost surely, and thus \cite[Proposition~1] {blanchet2024empirical} implies that $\EE[Y| X]=X$ almost surely. The claim now follows because $(X,Y)\sim\pi$.
\end{proof}

We now show when the generalized OT problem~\eqref{eq:mot} has a solution.

\begin{lemma}
\label{lemma:mot-attainment}
If \Cref{ass:growth} holds and $\mathds M(\nu, \hat \nu)< \infty$ for some $\nu,\hat\nu\in\mc P(\mc V,\mc W)$, then the infimum of the generalized OT problem~\eqref{eq:mot} is attained.
\end{lemma}

\begin{proof}
Select any $r>\mathds M(\nu, \hat \nu)$, and define the set of all couplings of $\nu$ and $\hat \nu$ as
\[
    \Pi(\nu, \hat \nu) = \left\{\pi\in \mc P((\mc V \times \mc W)^2): \pi_{(V, W)}  = \nu,\ \pi_{(\hat V, \hat W)}  = \hat \nu\right\}.
\]
It is well known that $\Pi(\nu, \hat \nu)$ is weakly compact; see, e.g., \cite[Corollary~3.16]{Kuhn_Shafiee_Wiesemann_2025}. Clearly, the infimum of problem~\eqref{eq:mot} does not change if we append the expected cost constraint $\EE_\pi[c((V, W), (\hat V, \hat W))] \leq r$. The set of all couplings $\pi\in \Pi(\nu, \hat \nu)$ that satisfy this constraint is also weakly compact because~$c$ is bounded below and lower semicontinuous, which implies via \cite[Proposition~3.3]{Kuhn_Shafiee_Wiesemann_2025} that $\EE_\pi[c((V, W), (\hat V, \hat W))]$ is weakly lower semicontinuous in~$\pi$. An elementary estimate then implies that any~$\pi\in \Pi(\nu, \hat \nu)$ with an expected cost of at most~$r$ satisfies the generalized moment bound
\begin{align*}
    \EE_\pi \bigl[\|(W,1)\|^{1+\delta}\bigr] & \leq  2^{\frac{1+\delta}{2}}\bigl(1+\EE_\pi \bigl[ |W|^{1+\delta}\bigr] \bigr) \\
    & \leq 2^{\frac{1+\delta}{2}}\left(1+ \frac{1}{k} \EE_\pi[c((V,W), (\hat V, \hat W))] + \frac{1}{k}\EE_{\hat\nu}\bigl[\varphi(\hat W)\bigr] \right) \\
    & \leq 2^{\frac{1+\delta}{2}}\left(1+ \frac{r}{k} + \frac{1}{k}\EE_{\hat\nu}\bigl[\varphi(\hat W)\bigr] \right)=M<\infty,
\end{align*}
where the second and fourth inequalities follow from Assumptions~\ref{ass:growth}\,(b) and~\ref{ass:growth}\,(a).

As the $\sigma$-algebra~$\mathcal G$ is countably generated (see Definition~\ref{def:mot}), one can use \cite[Theorem~2.1]{mackey1957borel} to prove that there exists a Borel function $\varrho:(\mathcal V\times\mathcal W)^2 \to [1,2]$ such that the random variable $R=\varrho(V,W,\hat V,\hat W)$ generates~$\mc G$. As~$R$ is strictly positive, the $\pi$-almost sure constraints $\EE_\pi[W | \mc G]=1$ and $\EE_\pi[W R| R]=R$ are equivalent. Definig the auxiliary random variables~$X=R$ and~$Y=WR$, the latter constraint simplifies to $\EE_\pi[Y|X]=X$. Our insights from the first part of the proof further imply that
\[
    \EE_{\pi} [\| (X,Y)\|^{1+\delta}] \leq 2^{1+\delta} \EE_\pi [\|(W,1)\|^{1+\delta}] \leq 2^{1+\delta} M
\]
on the family of all couplings~$\pi\in \Pi(\nu, \hat \nu)$ with an expected cost of at most~$r$. Thus, \Cref{lm:martinagle_weak_converge} ensures that the subset of all~$\pi$ in this family satisfying the $\pi$-almost sure constraint $\EE_\pi[Y|X]=X$ is weakly compact. In summary, we have shown that the restriction of the feasible set of problem~\eqref{eq:mot} to couplings with expected cost of at most~$r$ is weakly compact and that the objective function of~\eqref{eq:mot} is weakly lower semicontinuous. This implies that the infimum of problem~\eqref{eq:mot} is attained.
\end{proof}

\section*{Acknowledgment}
Jose Blanchet is supported by the Air Force Office of Scientific Research under award number FA9550-20-1-0397 and NSF 1915967, 2118199, 2229012, 2312204.
Daniel Kuhn is supported by the Swiss National Science Foundation under the NCCR Automation, grant agreement 51NF40\_180545. Jiajin Li is supported by a Natural Sciences and Engineering Research Council of Canada Discovery Grant RGPIN-2025-05817.

\bibliography{ref}
\bibliographystyle{abbrvnat}
\end{document}